\documentclass[11pt]{amsart}
\usepackage{amsmath,amstext,amsbsy,amssymb,amscd}
\usepackage{stmaryrd}
\usepackage{amsmath,mathrsfs}
\usepackage{amsxtra}
\usepackage{amscd}
\usepackage{amsthm}
\usepackage{amsfonts}
\usepackage{hyperref}
\usepackage{graphicx}%
\usepackage{amssymb}
\usepackage{eucal}
\usepackage{color}
\usepackage{multirow}
\usepackage[enableskew,vcentermath]{youngtab}
\hypersetup{colorlinks,linkcolor=blue,urlcolor=cyan,citecolor=blue}

\newtheorem{thm}{Theorem}[section]
\theoremstyle{plain}

\newtheorem{lem}[thm]{Lemma}
\newtheorem{prop}[thm]{Proposition}
\newtheorem{cor}[thm]{Corollary}

\theoremstyle{definition}

\newtheorem{example}[thm]{Example}

\theoremstyle{remark}
\newtheorem{rem}[thm]{Remark}
\newtheorem{conjecture}[thm]{Conjecture}

\definecolor{A}{rgb}{.75,1,.75}

\numberwithin{equation}{section}

\newcommand{\diag}{\text{diag}\,}

\newcommand{\Fq}{{\mathbb F}_q}

\newcommand{\GG}{\mathscr{G}}
\newcommand{\GL}{GL_n(q)}
\newcommand{\qbinom}[2]{\begin{bmatrix} #1\\#2 \end{bmatrix} }

\newcommand{\ds}{\displaystyle}

\newcommand{\C}{\mathbb C}
\newcommand{\Z}{\mathbb Z}
\newcommand{\N}{\mathbb N}

\newcommand{\al}{{\alpha}}

\newcommand{\mP}{\mathcal{P}}
\newcommand{\mA}{\mathcal{A}}

\newcommand{\la}{\lambda}
\newcommand{\bla}{{\boldsymbol{ \lambda}}}
\newcommand{\bmu}{{\boldsymbol{ \mu}}}
\newcommand{\bnu}{{\boldsymbol{ \nu}}}
\newcommand{\obla}{{\overline{\bla}}}
\newcommand{\obmu}{{\overline{\bmu}}}
\newcommand{\obmun}{{\bmu^{\uparrow n}}}
\newcommand{\obmuk}{{\bmu^{\uparrow k}}}
\newcommand{\obnu}{{\overline{\bnu}}}
\newcommand{\obnun}{{\bnu}^{\uparrow n}}

\newcommand{\LL}{{\bf q}}

\newcommand{\bg}{{\bar{g}}}
\newcommand{\bh}{{\bar{h}}}
\newcommand{\ZZ}{{\mathscr Z}}

{\vskip-\lastskip\medskip
  \noindent
  {\em #1.}\enspace
  }%
{\qed\par\medskip
  }

\begin{document}

\title[Stability of the centers of group algebras of $GL_n(q)$]{Stability of the centers of group algebras of $GL_n(q)$}
\author[Jinkui Wan and Weiqiang Wang]{Jinkui Wan and Weiqiang Wang}

\address{
(Wan) School of Mathematics, Beijing Institute of Technology, Beijing, 100081, P.R. China.}
\email{wjk302@hotmail.com}

\address{(Wang) Department of Mathematics, University of Virginia, Charlottesville,VA 22904, USA.}
\email{ww9c@virginia.edu}



\begin{abstract}
The center $\ZZ_n(q)$ of the integral group algebra of the general linear group $GL_n(q)$ over a finite field admits a filtration with respect to the reflection length. We show that the structure constants of the associated graded algebras $\mathscr{G}_n(q)$ are independent of $n$, and this stability leads to a universal stable center with positive integer structure constants which governs the algebras $\mathscr{G}_n(q)$  for all $n$. Various structure constants of the stable center are computed and several conjectures are formulated. Analogous stability properties for symmetric groups and wreath products were established earlier by Farahat-Higman and the second author.
\end{abstract}

\maketitle

\setcounter{tocdepth}{1}
 \tableofcontents

\section{Introduction}

\subsection{}

A remarkable stability for the centers of the integral group algebras $\Z[S_n]$ of the symmetric groups $S_n$ as $n$ varies was established by Farahat and Higman \cite{FH59}. This stability result can be reformulated conceptually as follows \cite{Ma95}. Define a notion of reflection length and modified type for permutations in $S_n$, so the length of a permutation is conjugation invariant and it is equal to the size of its modified type. The reflection length endows the center of $\Z[S_n]$ a filtered algebra structure; the stability result of Farahat-Higman states that the structure constants in the associated graded algebras  of the centers with respect to the basis of conjugacy class sums are independent of $n$. 
This stability result has led to a universal stable (Farahat-Higman) ring with a distinguished basis, which can be further identified with the ring of symmetric functions with a new basis \cite[pp.131-134]{Ma95}.

The above stability result has been generalized by the second author \cite{W04} to wreath products $\Gamma\wr S_n$ for any finite group $\Gamma$. 
When the group $\Gamma$ is a finite subgroup of $SL_2(\C)$, the associated graded algebra of the center of the group algebra of the wreath product is isomorphic to the cohomology ring of Hilbert scheme of $n$ points on the minimal resolution of $\C^2/\Gamma$; see \cite{W04}. (In case when $\Gamma$ is trivial, this goes back to \cite{LS01, Va01}.) The same type of stability results has been established in \cite{LQW04} for cohomology ring of Hilbert scheme of $n$ points on a large class of quasi-projective surfaces (conjecturally, on any non-projective surface).

\subsection{}

The general linear groups $GL_n(q)$ over a finite field $\Fq$ form another rich and sophisticated family of finite groups, which are often studied besides symmetric groups and wreath products; cf. \cite{Ma95, Ze81}. The main goal of this paper is to formulate and establish a stability result \`a la Farahat-Higman for the centers of the integral group algebras of $GL_n(q)$.

\subsection{}

An element in $\GL$ is called a reflection in this paper if its fixed point subspace in $\Fq^n$ has codimension one. The set of reflections in $\GL$ forms a generating set for $\GL$, and the reflection length of a general element $g\in\GL$ is by definition the length of any reduced word of $g$ in terms of reflections; two conjugate elements in $\GL$ have the same reflection length. The center $\ZZ_n(q)$ of the integral group algebra $\Z[\GL]$ of $\GL$ is a filtered algebra with a basis of conjugacy class sums with respect to the reflection length. Denote by $\mathscr{G}_n(q)$ the associated graded algebra.

Denote by $\Phi$ the set of monic irreducible polynomials in $\Fq[t]$ other than $t$. It is well known (cf. \cite{Ma95}) that the conjugacy classes of $\GL$ are parametrized by the types  $\bla =(\bla(f))_{f\in\Phi}\in \mP_n(\Phi)$ (which are the partition-valued functions on $\Phi$ of degree $n$; cf. \eqref{eq:deg-la}). We define a notion of modified types as follows. Let $g$ be an element of $GL_n(q)$ of type $\bla \in\mP_n(\Phi)$. Denote by $\bla^e =\bla(t-1)$ the partition of the unipotent Jordan blocks, and denote by $r=\ell(\bla^e)$ its length. We define the {\em modified type } of $g$ to be $\mathring{\bla}\in\mP_{n-r}(\Phi)$, where $\mathring{\bla}(f)=\bla(f)$ for $f\neq t-1$ and $\mathring{\bla}(t-1)=(\bla^e_1-1,\bla^e_2-1,\ldots,\bla^e_r-1)$.
This modified type remains unchanged for $g$ under the embedding of $GL_n(q)$ into $GL_{n+1}(q)$ and it is also clearly conjugation invariant. It follows that the conjugacy classes of $GL_\infty(q)$ are parametrized by the modified types in $\mP(\Phi) =\cup_n \mP_n(\Phi)$.

As observed in \cite{HLR17}, a basic property about the reflection length of $g\in \GL$ is that it coincides with the codimension of its fixed point subspace in $\Fq^n$. We show that the reflection length of an element $g\in \GL$ is equal to the size of its modified type.

We parametrize the conjugacy classes (and class sums) for $\GL$ via the modified types $\bla$, and denote the conjugacy classes by $\mathscr K_\bla (n)$ and the corresponding class sums by $K_\bla(n)$, for $\|\bla \| + \ell(\bla^e) \le n$.
We then write the multiplication in the center $\ZZ_n(q)$ as
\begin{equation}
  \label{eq:a(n)}
K_\bla(n) K_\bmu(n)=\sum_{\bnu:\; \|\bnu\|\leq \|\bla\|+\|\bmu\|} a^\bnu_{\bla\bmu}(n)K_\bnu(n).
\end{equation}

We can now state our first main result of this paper.
\begin{thm} [Theorem~\ref{thm:indep}]\label{thm:intro}
Let $\bla,\bmu,\bnu\in\mP(\Phi)$.
If $\|\bnu\|= \|\bla\|+\|\bmu\|$,
then $a^\bnu_{\bla\bmu}(n)$ is independent of $n$. 
{\em (In this case, we shall write $a^\bnu_{\bla\bmu}(n)$ as $a^\bnu_{\bla\bmu} \in \N$.)}
\end{thm}

After we proved Theorem~\ref{thm:indep}, we found a paper by M\'eliot, in which the structure constants $a^\bnu_{\bla\bmu}(n)$ in \eqref{eq:a(n)} for the centers $\ZZ_n(q)$ were studied. Inspired by Kerov-Ivanov's partial permutations,  M\'eliot \cite{M14} developed a very interesting notion of partial isomorphisms for $\GL$ and used it to show that $a^\bnu_{\bla\bmu}(n)$ are polynomials in $x$ evaluated at $x=q^n$; see however Remark \ref{rem:Meliot}.
This can be viewed as an analogue of another result of Farahat-Higman for symmetric groups.
The concepts of reflection length filtration of $\ZZ_n(q)$ and modified types were not present {\em loc. cit.} however, and
the parametrization of the class sums for $\GL$ therein often uses $\bmu$ with $\bmu^e$ containing no part equal to $1$.
Our paper provides an in depth study of these structure constants complementary to \cite{M14},
focusing on arguably the more interesting and accessible ones.

A key ingredient in the proof of Theorem~\ref{thm:intro} is the existence of a normal form of triples in the following sense (see Proposition~\ref{prop:length-eq}).
Assume $\|\bnu\|=\|\bla\|+\|\bmu\|$. Any triple of elements $(g, h, gh)$ in $GL_\infty(q)$ of modified types $\bla, \bmu$  and $\bnu$ is conjugate (under the simultaneous conjugation of $GL_\infty(q)$) to some triple $(\bg, \bh, \bg\bh)$
of elements in $GL_k(q)$ with $k=\|\bnu\| +\ell( \bnu^e)$, where we regard $GL_k(q)$ naturally as a subgroup of $GL_\infty(q)$.

Theorem~\ref{thm:intro} can be rephrased as that the associated graded algebra $\mathscr{G}_n(q)$ of $\ZZ_n(q)$ has structure constants independent of $n$. We introduce a graded $\mathbb{Z}$-algebra $\mathscr{G}$ with a basis given by the symbols $K_\bla$ indexed by $\bla\in\mP(\Phi)$, and its multiplication has structure constants $a^\bnu_{\bla\bmu}$ as in the theorem above, for $\|\bnu\|= \|\bla\|+\|\bmu\|$; cf. \eqref{eq:KKaK}.

\begin{thm} [Theorem~\ref{thm:stable}]
The graded $\mathbb{Z}$-algebra $\mathscr{G}_n(q)$ has the multiplication given by
$$
K_\bla(n) K_\bmu(n)=\sum_{\|\bnu\|=\|\bla\|+\|\bmu\|}a^\bnu_{\bla\bmu}K_\bnu(n),
$$
for $\bla,\bmu\in\mP(\Phi)$. Moreover, we have a surjective homomorphism $\mathscr{G}\twoheadrightarrow\mathscr{G}_n(q)$ for each $n$, which maps $K_\bla$ to $K_\bla(n)$ for all $\bla\in\mP(\Phi)$.
\end{thm}

\subsection{}
We conjecture that the stable center $\mathbb Q\otimes_{\Z} \mathscr{G}(q)$ is the polynomial algebra generated by the single cycle class sums; for a precise formulation see Conjecture~\ref{conj:generators}. Similar structure results hold for the stable centers in the settings of symmetric groups and wreath products; cf. \cite{FH59, W04}.

The computations of the structure constants  $a^\bnu_{\bla\bmu}$ are much more difficult than in the symmetric group or wreath product settings; cf. \cite{FH59, W04}. We compute various examples of these structure constants (see Theorem~\ref{thm:2reflections}, Proposition~\ref{prop:a-union}, Proposition~\ref{prop:a-union-equal}) in Section~\ref{sec:computation}.

All examples indicate
a phenomenon (which is rather striking to us)
that these structure constants $a^\bnu_{\bla\bmu}$ only depend on the configurations but not on the precise supports  of the modified types $\bla, \bmu, \bnu$; see \S\ref{subsec:discussions}(5) for a precise formulation.  The examples have also motivated several conjectures on more general structure constants in Section~\ref{sec:Conj}, where we also discuss a few open problems and further directions which arise from this work.

In particular, we ask to what extent the structure constants (not merely the stable ones) are polynomials in $q$; this is a little subtle as the indexing sets for the structure constants rely on the conjugacy classes of $GL_n(q)$ which depend on $q$. We offer a possible formulation of
{\em generic} structure constants; see Conjecture~\ref{conj:intpoly}  and \S\ref{subsec:discussions}(6).

%
\subsection{}
The paper is organized as follows.
In Section~\ref{sec:GL}, we review and set up notations for conjugacy classes and their canonical representatives of $\GL$. We introduce the notion of modified types.
In Section~\ref{sec:stable}, we formulate and establish the stability on the structure constants for the graded algebra $\GG_n(q)$ and the universal stable center $\GG(q)$.
In Section~\ref{sec:computation}, we compute various structure constants for $\GG(q)$. We formulate a few conjectures and further research directions in Section~\ref{sec:Conj}.

 \vspace{.4cm}
{\bf Acknowledgements.}
It is a pleasure to thank Xuhua He for his stimulating questions which are answered here. Various examples were computed using Sage, and we thank Arun Kannan for Sage tutorial. This project was carried out while the first author enjoyed the support and hospitality of University of Virginia and Institute of Mathematical Science, and she is partially supported by  NSFC-11571036 and China Scholarship Council 201706035027.  The second author is partially supported by the NSF grant DMS-1702254.

\section{Conjugacy classes and centralizers in $GL_n(q)$}
  \label{sec:GL}

In this section, we will review the conjugacy classes of the general linear group $GL_n(q)$ and set up notations (cf. \cite{Ma95}). We provide a description of the centralizers of representatives of these conjugacy classes. A notion of modified types is introduced and used to parametrize the conjugacy classes of $GL_n(q)$ and $GL_\infty(q)$.

\subsection{Conjugacy classes of $GL_n(q)$}

Denote by $\mP$ the set of all partitions. For $\la=(\la_1, \la_2, \ldots,)\in\mP$, we denote its size by $|\la|=\la_1+\la_2+\cdots+\la_{\ell}$, its length by $\ell(\la)$, and also denote
\[
n(\la)=\sum_{i\geq 1}(i-1)\la_i.
\]
We will also write $\la=(1^{m_1(\la)}2^{m_2(\la)} \ldots)$, where $m_i(\la)$ is the number of parts in $\la$ equal to $i$.
For two partitions $\la,\mu\in\mP$, we denote by $\la\cup\mu$ the partition whose parts are those of $\la$ and $\mu$. For a set $Y$, let $\mP(Y)$ be the set of the partition-valued functions $\boldsymbol{\la}:Y\rightarrow \mP$ such that only finitely many $\bla(y)$ are nonempty partitions. Given $\bla,\bmu\in\mP(Y)$, we define $\bla\cup\bmu\in\mP(Y)$ by letting $(\bla\cup\bmu) (y) =\bla(y) \cup\bmu(y)$ for each $y\in Y$.

Denote by $\Fq$ the finite field of $q$ elements, where $q$ is a prime power. We shall regard vectors in the $n$-dimensional vector space $\Fq^n$ as column vectors, that is, $\Fq^n=\{v=(v_1,\ldots,v_n)^\intercal|v_k\in\Fq, 1\leq k\leq n\}$
for each $n\geq 1$. Denote by $M_{n\times m}(q)$ the set of $n\times m$ matrices over the finite field $\Fq$.
The general linear group $GL_n(q)$, which consists of all invertible matrices in $M_{n\times n}(q)$, acts on $\Fq^n$ naturally via left multiplication. We  shall abbreviate $GL_n(q)$ as $G_n$.

The conjugacy classes of $G_n$ can be described as follows (cf. \cite{Ma95}). For $g,h\in G_n$, write $g\sim h$ if $g$ is conjugate to $h$. Each element $g\in G_n$ acts on the vector space $\Fq^n$ and hence defines a $\Fq[t]$-module on $\Fq^n$ such that $tv=gv$ for $v\in\Fq^n$. Denote this $\Fq[t]$-module by $V_g$. Then $g\sim h$ if and only if $V_g\cong V_h$ as $\Fq[t]$-modules. Hence the conjugacy classes of $G_n$ are in one to one correspondence with the isomorphism classes of $\Fq[t]$-modules $V$ such that $\dim V=n$ and if $tv=0$ then $v=0$ for $v\in V$. Since $\Fq[t]$ is a principal ideal domain, each $\Fq[t]$-module is isomorphic to a direct sum of cyclic modules of the form $\Fq[t]/(f)^m$, where $m\geq 1$, $f\in\Fq[t]$ is a monic irreducible polynomial and $(f)$ is the ideal generated by $f$.

Let $\Phi$ be the set of all monic irreducible polynomial in $\Fq[t]$ other than $t$. Then for each $g\in G_n$, there exists a unique $\bla =(\bla(f))_{f\in\Phi} \in\mathcal{P}(\Phi)$ such that
\begin{equation}
   \label{eq:Vg}
V_g\cong V_\bla:=\oplus_{f,i}\Fq[t]/(f)^{\bla_i(f)},
\end{equation}
where we write $\bla(f)=(\bla_1(f),\bla_2(f),\ldots)\in\mP$; moreover, we have
\begin{equation}
    \label{eq:deg-la}
\|\bla\|:=\sum_{f\in\Phi}d(f)|\bla(f)|=n,
\end{equation}
where $d(f)$ denotes the degree of the polynomial $f$. Denote by $\mP_n(\Phi)$ the set of $\bla\in\mP(\Phi)$ satisfying \eqref{eq:deg-la}.
The partition-valued function $\bla=(\bla(f))_{f\in\Phi}\in\mP(\Phi)$ is called the {\em type} of $g$. Then any two elements of $G_n$ are conjugate if and only if they have the same type, and there is a bijection between the set of conjugacy classes of $G_n$ and the set $\mP_n(\Phi)$. 

For each $f=t^d-\sum_{1\leq i\leq d}a_it^{i-1}\in\Phi$, let $J(f)$ denote the {\em companion matrix} for $f$ of the form
\[
J(f)=
  \begin{bmatrix}
   0 & 1& 0 & \cdots & 0\\
   0 & 0& 1 & \cdots & 0\\
   \vdots & \vdots& \vdots & \vdots & \vdots\\
   0 & 0& 0 & \cdots & 1\\
   a_1 & a_2& a_3 & \cdots & a_d
  \end{bmatrix},
\]
and for each integer $m\geq 1$ let
\[
J_m(f)=\begin{bmatrix}
   J(f) & I_d& 0 & \cdots & 0&0\\
   0 & J(f)& I_d & \cdots & 0&0\\
   \vdots & \vdots& \vdots & \vdots & \vdots\\
   0 & 0& 0 & \cdots & J(f) & I_d\\
   0 & 0& 0 & \cdots & 0& J(f)
  \end{bmatrix}_{dm\times dm}
\]
with $m$ diagonal blocks $J(f)$, where $I_d$ is the $d\times d$ identity matrix.
Given $\bla\in \mP(\Phi)$ with $\bla(f)=(\bla_1(f),\bla_2(f),\ldots)$,
set
\begin{equation}\label{eq:Jbla}
J_\bla=\diag\Big(J_{\bla_i(f)}(f)\Big)_{f,i},
\end{equation}
that is, $J_\bla$ is the diagonal sum of the matrices $J_{\bla_i(f)}(f)$ for all $i\geq 1$ and $f\in\Phi$.
Then an element $g\in G_n$ of type $\bla$ is conjugate to the canonical form $J_\bla$ (cf. \cite[Chapter IV, \S 2]{Ma95}).
For $f\in \Phi$, set
\[
J_\bla(f)=\diag\Big(J_{\bla_i(f)}(f)\Big)_{i\geq1}.
\]
Then by \eqref{eq:Vg}, we have
$
V_{J_\bla(f)}\cong \oplus_{i\geq 1}\Fq[t]/(f)^{\bla_i(f)}
$
as $\Fq[t]$-modules and moreover, we have $
J_\bla=\diag\big(J_\bla(f)\big)_{f\in\Phi}.$

\begin{lem}[{cf. \cite[IV, (2.5)]{Ma95}}]\label{lem:comm-Jf1f2}
Let $\bla,\bmu\in\mP(\Phi)$ and $f_1\neq f_2\in\Phi$.
Suppose $A$ is a $d(f_2)|\bmu(f_2)|\times d(f_1)|\bla(f_1)|$-matrix over $\Fq$
satisfying
$
A J_\bla(f_1)=J_\bmu(f_2)A.
$
Then $A=0$.
\end{lem}

For any partition $\la\in\mP$, define
\[
a_\la(q)=q^{|\la|+2n(\la)}\prod_{i\geq 1}\varphi_{m_i(\la)}(q^{-1}),
\]
where for $k\geq 0$, we have denoted
\[
\varphi_k(t)=(1-t)(1-t^2)\cdots(1-t^k).
\]
For a polynomial $f\in\Phi$, we set
\begin{equation}
  \label{eq:qf}
q_f=q^{d(f)}.
\end{equation}
For $\bla\in\mP_n(\Phi)$, denote by $\mathcal{A}_\bla$ the centralizer of the element $J_\bla$ in $G_n$.
It is known (cf. \cite[II, (1.6)]{Ma95}) that the centralizer of the element $J_\bla(f)$ in $GL_{d(f)|\bla(f)|}(q)$
has the order $a_{\bla(f)}(q_f)$  and hence by Lemma \ref{lem:comm-Jf1f2} the centralizer of an element $g\in G_n$ of type $\bla$
has order
\begin{equation}\label{eq:cc-ord}
|\mathcal{A}_\bla|=\prod_{f\in\Phi}a_{\bla(f)}(q_f).
\end{equation}

\subsection{The group $G_\infty$}

For $m\le n$, by the natural identification
\[
V_m =\big\{(a_1, \ldots, a_m, \underbrace{0,\ldots, 0}_{n-m})^\intercal \big| a_i \in \Fq \big\},
\]
we regard $V_m$ as a subspace of $V_n$ (denoted by $V_m \subseteq V_n$).
So we have a natural filtration of vector spaces
$
0=V_0\subset V_1\subset \cdots \subset V_{n} \subset V_{n+1} \subset \cdots.
$ 
We also denote
\[
V'_{n-m} = \big\{(\underbrace{0,\ldots, 0}_m, b_1, \ldots, b_{n-m})^\intercal \big| b_i \in \Fq \big\},
\]
another distinguished subspace of $V_n$ of dimension $n-m$.
Accordingly, via the embedding $g \mapsto  
\begin{bmatrix} g & 0\\0 & I_{n-m} \end{bmatrix}$, we regard $G_m$ as a subgroup of $G_n$. In this way we have a natural filtration of groups
\[
1= G_0\leq G_1\leq \cdots \leq G_{n} \leq G_{n+1} \leq \cdots.
\]
Then the union
$
G_\infty=\cup_{n\geq 0}G_n
$
carries a natural group structure.

\subsection{The modified type}

For $\bla\in\mP(\Phi)$, we introduce a shorthand noation
\[
\bla^e:=\bla(t-1).
\]
Let $g$ be an element of $G_n$ of type $\bla=(\bla(f))_{f\in\Phi}\in\mP_n(\Phi)$.
If we regard $g$ as an element in $G_{n+m}$ by the natural embedding $G_n\subset G_{n+m}$ for any $m\ge 1$, then the type of $g$ changes.
We define the {\em modified type } of $g$ to be $\mathring{\bla}\in\mP_{n-r}(\Phi)$, where $r=\ell(\bla^e)$
and $\mathring{\bla}(f)=\bla(f)$ for $f\neq t-1$ and $\mathring{\bla}(t-1)=(\bla^e_1-1,\bla^e_2-1,\ldots,\bla^e_r-1)$.
This modified type is the same for a given element under the embedding of $G_n$ in $G_{n+m}$. (The notion of modified types here is inspired by an analogous notion for symmetric groups (cf. \cite[p. 131]{Ma95}) and wreath products \cite[\S2.3]{W04}.)
The following is immediate.

\begin{lem}
Two elements in $G_\infty $ are conjugate if and only if they have the same modified type.
\end{lem}

Given $\bmu \in\mP(\Phi)$ with $r=\ell(\bmu^e)$ and $\bmu^e=(\bmu^e_1,\bmu^e_2,\ldots,\bmu^e_r)$,
we define $\bmu^{\uparrow n} \in \mP_n(\Phi)$ for all $n\geq \|\bmu\|+r$
via
\begin{align}
\label{eq:obmu}
\bmu^{\uparrow n}(f)&=\bmu(f), \; {\rm for }\, f\neq t-1,
\\
  \label{eq:bmuarr}
  \bmu^{\uparrow n}(t-1)&=(\bmu^{\uparrow n})^{e} =(\bmu^e_1+1,\bmu^e_2+1,\ldots,\bmu^e_r+1,\underbrace{1,\ldots,1}_{n-r-\|\bmu\|}).
\end{align}
Clearly elements of type $\bmu^{\uparrow n}$ in $G_n$ have a modified type $\bmu$.

Given $\bmu\in\mP(\Phi)$, we denote by $\mathscr{K}_\bmu$ the conjugacy class in $G_\infty$ which consists of elements of modified type  $\bmu$. For each $\bmu\in\mP(\Phi)$, $\mathscr K_\bmu(n):= G_n\cap\mathscr{K}_\bmu$ is nonempty if and only if $\|\bmu\|+\ell(\bmu^e)\leq n$; in this case $\mathscr K_\bmu(n)$ is a conjugacy class of $G_n$.
Let $K_\bmu(n)$ be the class sum of $\mathscr K_\bmu(n)$ if $\|\bmu\|+\ell(\bmu^e)\leq n$, and be 0 otherwise. Denote by $\ZZ_n(q)$ the center of the integral group algebra $\mathbb{Z}[G_n]$. We summarize these discussions in the following.

\begin{lem}
   \label{lem:Kbmu}
The set $\{K_\bmu(n) \neq 0 | \bmu \in \mathcal{P}(\Phi)\}$  forms the class sum $\mathbb{Z}$-basis for the center $\ZZ_n(q)$, for each $n\ge 0$.
\end{lem}

\subsection{The centralizers}

Recall $J_m(t-1)$ is the Jordan form of size $m$ and eigenvalue $1$. The following elementary lemma (cf. \cite[Lemma 2.1]{CW18}) can be verified by a direct computation.

\begin{lem}
  \label{lem:comm-Jkm}
Let $k,m\geq 1$. Suppose $A\in M_{m\times k}(q)$ satisfies $AJ_k(t-1) =J_m(t-1)A$. Then $A$
is of the form
\[
A=\begin{bmatrix}
 0      & \cdots  & 0      & a_1     & a_2        &\cdots     &  a_{m-1}             & a_m                 \\
 0      & \cdots  & 0      & 0       & a_1       &\cdots      & a_{m-2}              &a_{m-1}           \\
 \vdots &         & \vdots & \vdots  &            & \ddots    & \vdots               & \vdots             \\
 0      & \cdots  & 0      & 0       &  0         &\cdots     & a_1                  & a_2       \\
 0      & \cdots  & 0      & 0       & 0          &\cdots     & 0                    & a_1
\end{bmatrix} \quad \text{if } m\leq k,
\]
or
\[
 A=\begin{bmatrix}
  a_1     & a_2      &    \cdots       & a_{k-1}           &a_k                \\
   0      &a_1       & \cdots          & a_{k-2}          & a_{k-1}            \\
          &          &  \ddots         &                   &                    \\
   0      &    0     &                 &  a_1              &     a_2       \\
   0      &     0    &   \cdots        &     0             &   a_1    \\
   0      &    0     &   \cdots        &   0               &  0                  \\
 \vdots   &\vdots    &                 &\vdots             &\vdots               \\
  0       &    0     &    \cdots       &    0              &  0
\end{bmatrix} \quad \text{if } m\geq k,
\]
for some scalars $a_1,\ldots,a_{\min(k,m)}\in\Fq$.
\end{lem}

The order of the centralizer of a given element of type $\bla$ in $G_n$ is known; cf. \eqref{eq:cc-ord}. For our purpose, we need to have a more precise description of the centralizer. Recall $\bmu^{\uparrow n}$ from \eqref{eq:obmu} and \eqref {eq:bmuarr}.

\begin{prop}
   \label{prop:A-pbmu}
Let $n\geq 0, \bmu\in\mP(\Phi)$. Suppose $k=\|\bmu\|+\ell(\bmu^e)\leq n$.
Then the centralizer  of $J_{\obmun} \in G_n$ is given by
\begin{equation}
   \label{eq:centralizer}
\mathcal{A}_{\obmun}=\bigg\{
  \begin{bmatrix}
   A & B\\
   C & D
  \end{bmatrix}
\Big| A \in \mathcal{A}_{\obmuk}, D\in G_{n-k}, J_{\obmuk}B=B, CJ_{\obmuk}=C \bigg\}.
\end{equation}
\end{prop}
\begin{proof}
By \eqref{eq:Jbla}, \eqref{eq:obmu} and \eqref{eq:bmuarr}, we have
$$
J_{\obmun}=
  \begin{bmatrix}
   J_{\obmuk} & 0\\
   0 & I_{n-k}
  \end{bmatrix}.
$$
Write $P\in\mA_\obmun$ in a $(k|n-k)$-block form as
\[
P=
  \begin{bmatrix}
   A & B\\
   C & D
  \end{bmatrix}.
\]
Then we have $PJ_\obmun=J_\obmun P$ if and only if
\begin{equation}\label{eq:ABC}
AJ_\obmuk =J_\obmuk A, \quad J_\obmuk B=B, \quad CJ_\obmuk=C.
\end{equation}
Comparing with the right hand side of \eqref{eq:centralizer}, it remains to show the following.

\vspace{2mm}
\noindent {\bf Claim 1.} A matrix
$  \begin{bmatrix}
   A & B\\
   C & D
  \end{bmatrix} $
satisfying \eqref{eq:ABC} is invertible if and only if both $A$ and $D$ are invertible.

Let us reduce Claim 1 to a special case. Let
$
  \begin{bmatrix}
   A & B\\
   C & D
  \end{bmatrix}
\in G_n$
be such that \eqref{eq:ABC} holds.
By \eqref{eq:Jbla}, we can write
$$
J_\obmuk=\text{diag}\big(J_{\obmuk(f)}\big)_{f\in\Phi}=\text{diag}\Big(\text{diag}(J_\obmuk(f))_{f\neq t-1}, J_{(\obmuk)^e}\Big).
$$
Then by Lemma \ref{lem:comm-Jf1f2} and \eqref{eq:ABC}, we can write
$$
\begin{bmatrix}
   A & B\\
   C & D
\end{bmatrix}
=
\begin{bmatrix}
   A_1 & 0   & 0\\
   0   & A_2 & B_1\\
   0   & C_1   & D
  \end{bmatrix},
$$
where
$A=
  \begin{bmatrix}
   A_1 & 0   \\
   0   & A_2
  \end{bmatrix},
  B=
  \begin{bmatrix}
   0\\
   B_1
  \end{bmatrix},
   C=
  \begin{bmatrix}
   0   & C_1
  \end{bmatrix}
$
satisfy
\begin{align}
A_1\text{diag}\big(J_{\obmuk(f)}\big)_{f\neq t-1} &=\text{diag}\big(J_{\obmuk(f)}\big)_{f\neq t-1} A_1,
\notag
 \\
A_2 J_{(\obmuk)^e}=J_{(\obmuk)^e} A_2, \qquad
J_{(\obmuk)^e}B_1 &=B_1, \qquad
C_1J_{(\obmuk)^e}=C_1.
\label{eq:A2B1C1}
\end{align}
Clearly $
  \begin{bmatrix}
   A_1 & 0   & 0\\
   0   & A_2 & B_1\\
   0   & C_1   & D
  \end{bmatrix}
$ is invertible if and only if both $A_1$ and
$
  \begin{bmatrix}
  A_2 & B_1\\
   C_1   & D
  \end{bmatrix}
$
are invertible.
Thus Claim~1 is reduced to the following special case when $\bmu(f)=\emptyset$ for all $f\neq t-1$.

\vspace{2mm}
\noindent {\bf Claim 2.}
A matrix
$
M:= \begin{bmatrix}
A_2 & B_1\\
C_1   & D
  \end{bmatrix} $
  satisfying \eqref{eq:A2B1C1} is invertible if and only if both $A_2$ and $D$ are invertible.

Let us prove Claim 2. Thanks to  \eqref{eq:bmuarr}, we can write $(\obmuk)^e=(\obmu_1, \obmu_2, \cdots, \obmu_r)$, with $\obmu_r\geq 2$.
Then by \eqref{eq:A2B1C1} and Lemma \ref{lem:comm-Jkm}, we can write
\begin{equation}\label{eq:A2B1C1D}
M =
  \begin{bmatrix}
A_{11} & A_{12} &\cdots & A_{1r} & B_{11}\\
A_{21} & A_{22} & \cdots & A_{2r} & B_{21}\\
\vdots&\vdots&\vdots&\vdots&\vdots\\
A_{r1} & A_{r2} &\cdots & A_{rr}& B_{r1}\\
C_{11} & C_{12} & \cdots & C_{1r}& D
  \end{bmatrix},
\end{equation}
where $A_{ij}$ are of the form
\begin{equation}\label{eq:Aij-1}
A_{ij}=\begin{bmatrix}
 0      & \cdots  & 0      & a_1     & a_2        &\cdots     &  a_{\obmu_i-1}             & a_{\obmu_i}                 \\
 0      & \cdots  & 0      & 0       & a_1       &\cdots      & a_{\obmu_i-2}              &a_{\obmu_i-1}           \\
 \vdots &         & \vdots & \vdots  &            & \ddots    & \vdots               & \vdots             \\
 0      & \cdots  & 0      & 0       &  0         &\cdots     & a_1                  & a_2       \\
 0      & \cdots  & 0      & 0       & 0          &\cdots     & 0                    & a_1
\end{bmatrix}_{\obmu_i\times\obmu_j}  \quad \text{if } i\geq j,
\end{equation}
or
\begin{equation}\label{eq:Aij-2}
 A_{ij}=\begin{bmatrix}
  a_1     & a_2      &    \cdots       & a_{\obmu_j-1}           &a_{\obmu_j}                \\
   0      &a_1       & \cdots          & a_{\obmu_j-2}          & a_{\obmu_j-1}            \\
          &          &  \ddots         &                   &                    \\
   0      &    0     &                 &  a_1              &     a_2       \\
   0      &     0    &   \cdots        &     0             &   a_1    \\
   0      &    0     &   \cdots        &   0               &  0                  \\
 \vdots   &\vdots    &                 &\vdots             &\vdots               \\
  0       &    0     &    \cdots       &    0              &  0
 \end{bmatrix}_{\obmu_i\times\obmu_j} \quad \text{if } i\leq j,
\end{equation}
for some scalars $a_1,\ldots,a_{\min(\obmu_i,\obmu_j)}\in\Fq$,
and $B_{i1}, C_{1j}$ are of the form
\begin{equation}\label{eq:Bi1}
B_{i1}=
\begin{bmatrix}
b_1   &    b_2   &   \cdots   &    b_{n-k}\\
0     &     0    &   \cdots    &    0\\
\vdots&\vdots&\vdots&\vdots\\
0&0&\cdots&0
\end{bmatrix}_{\obmu_i\times (n-k)},
\qquad
C_{1j}=
\begin{bmatrix}
0        &     \cdots    &  0       &      c_1\\
0        &    \cdots     &  0       &      c_2\\
\vdots   &    \vdots     &  \vdots  &     \vdots\\
0        & \cdots        &   0      & c_{n-k}
\end{bmatrix}_{(n-k) \times\obmu_j},
\end{equation}
for some scalars $b_i, c_i \in \Fq$, where $1\le i \le n-k$.

Denote by $\obmu_{1..i} =\obmu_1+\obmu_2 +\ldots+\obmu_i$ for $1\le i \le r$.
Then from Equations~ \eqref{eq:Aij-1}--\eqref{eq:Bi1} and the fact that $\obmu_i\geq 2$ for $1\leq i\leq r$,
we make the following observations: (I)  The nonzero elements in the rows $\obmu_{1}, \obmu_{1..2}, \ldots, \obmu_{1..r}$ in the matrix $M$ all lie in the  columns $\obmu_{1}, \obmu_{1..2}, \ldots, \obmu_{1..r}$. (II) The nonzero elements in the matrix
$C=
\begin{bmatrix}
C_{11}       &     C_{12}    &  \ldots       &      C_{1r}
\end{bmatrix}
$
all lie in the columns $\obmu_{1}, \obmu_{1..2}, \ldots, \obmu_{1..r}$. 

Denote by $E$ the submatrix of $A_2$ of rows/columns  $\obmu_{1}, \obmu_{1..2}, \ldots, \obmu_{1..r}$, and denote by $A_2'$ the submatrix of $A_2$ with rows/columns  $\obmu_{1}, \obmu_{1..2}, \ldots, \obmu_{1..r}$ removed. Then the submatrix of $M$ with rows/columns  $\obmu_{1}, \obmu_{1..2}, \ldots, \obmu_{1..r}$ removed is of the form $\begin{bmatrix} A_2' & * \\ 0 & D \end{bmatrix}$. Applying the Laplace expansion formula along the rows $\obmu_{1}, \obmu_{1..2}, \ldots, \obmu_{1..r}$ to compute the determinants $\det M$ (or $\det A_2$) only produces one nontrivial term, thanks to the observations (I)--(II) above. Hence we have
\begin{align*}
\det
  \begin{bmatrix}
A_2 & B_1\\
C_1   & D
  \end{bmatrix}
= \det E \cdot \det \begin{bmatrix} A_2' & * \\ 0 & D \end{bmatrix}
= \det E \cdot \det  A_2'  \cdot \det D
= \det A_2\cdot \det D.
\end{align*}
Therefore the matrix
$
  \begin{bmatrix}
A_2 & B_1\\
C_1   & D
  \end{bmatrix}
$
is invertible if and only if both $A_2$ and $D$ are invertible. This proves Claim~2 and hence completes the proof of Claim~1.

The proposition is proved.
\end{proof}

\begin{rem}
The centralizers of a different set of representatives for the conjugacy classes of $G_n$ can be precisely described, following a variant of \cite[Lemma~4.8 and its proof]{CW18}. This can in particular provide another proof of Proposition~\ref{prop:A-pbmu}. We will skip the details as we do not need such a result in this paper.
\end{rem}

\begin{cor}\label{cor:invertible}
Suppose $g\in G_n$ is of the form
$g=
  \begin{bmatrix}
   \bg & 0\\
   0 & I_{n-k}
  \end{bmatrix}
$ and the type of $\bg\in G_k$ is $\bla$ for some $0\leq k\leq n$.
If all parts of the partition $\bla^e$ are strictly bigger than 1, then
a $(k|n-k)$-block matrix
$\begin{bmatrix}
   A & B\\
   C & D
  \end{bmatrix}
$ commuting with $g$ is invertible if and only if both $A$ and $D$ are invertible.
\end{cor}

\begin{proof}
Since the type of $\bg$ is $\bla$, there exists $h\in G_k$ such that $h\bg h^{-1}=J_\bla$.
Suppose
$\begin{bmatrix}
   A & B\\
   C & D
  \end{bmatrix}
$
commutes with $g$. Set
$
H=\begin{bmatrix}
h & 0\\
0 & I_{n-k}
\end{bmatrix}.
$
Then we have
\begin{align*}
\begin{bmatrix}
h^{-1}Ah & h^{-1}B\\
Ch & D
\end{bmatrix}
&= H^{-1}
\begin{bmatrix}
   A & B\\
   C & D
  \end{bmatrix}
H,
\\
\begin{bmatrix}
J_\bla & 0\\
0 & I_{n-k}
\end{bmatrix}
& =H^{-1} g H,
\end{align*}
and these two matrices commute with each other.
By Proposition \ref{prop:A-pbmu} and the assumption that all parts of $\bla^e$ are strictly bigger than $1$, the matrix
$
\begin{bmatrix}
h^{-1}Ah & h^{-1}B\\
Ch & D
\end{bmatrix}
$
is invertible if and only if both $hAh^{-1}$ and $D$ are invertible.
The corollary follows.
\end{proof}

We record the following corollary for later use.

\begin{cor}\label{cor:Kmu-n}
Let $\bmu\in\mP(\Phi)$. Suppose $k=\|\bmu\|+\ell(\bmu^e) \leq n$.
Then the cardinality of the centralizer $\mA_{\obmun}$ of $J_{\obmun}$ in $G_n$ is equal to
\begin{align}\label{eq:ccmu-n}
\big|\mA_{\obmuk}\big|\cdot \big|G_{n-k}\big|\cdot
\big|\{B\in M_{k\times (n-k)}(q)|J_{\obmuk} B=B\}\big|\cdot
\big|\{C\in M_{(n-k)\times k}(q)| C J_{\obmuk} =C\}\big|.
\end{align}
\end{cor}

\begin{rem} Observe that Corollary \ref{cor:Kmu-n} provides another interpretation of the cardinality of $\mA_{\obmun}$.
It can be compared to the general formula given by \eqref{eq:cc-ord} in the following way.
Let $\bmu\in\mP(\Phi)$ with $r=\ell(\bmu^e)$. Suppose $k=\|\bmu\|+r \leq n$. Then by \eqref{eq:obmu} and \eqref{eq:bmuarr}, we observe that
$
a_{\obmun(f)}(q_f)=a_{\obmuk(f)}(q_f), \text{ for } f\neq t-1
$ and furthermore $m_i((\obmun)^e)=m_i((\obmuk)^e)$ for $i\geq 2$ and $m_1(\obmuk)=0, m_1(\obmun)=n-k$.
This together with \eqref{eq:cc-ord} implies
\begin{align}
\big|\mA_{\obmun}\big|&=\big(\prod_{f\in\Phi}a_{\obmuk(f)}(q_f)\big)\cdot q^{2r(n-k)}(q^{n-k}-1)(q^{n-k}-q)\cdots(q^{n-k}-q^{n-k-1})\notag\\
&=\big|\mA_{\obmuk}\big|\cdot \big|G_{n-k}\big|\cdot q^{2r(n-k)}. \label{eq:ccmu-n-decom}
\end{align}
On the other hand, since $\text{rank}(J_{\obmuk}-I_k)=n-r$, we obtain
$$
\big|\{B\in M_{k\times (n-k)}(q)|J_{\obmuk} B=B\}\big|=\big|\{B\in M_{k\times (n-k)}(q)|(J_{\obmuk}-I_k) B=0\}\big|=q^{r(n-k)},
$$
where $I_k$ is the $k\times k$ identity matrix. Similarly, $\big|\{C\in M_{(n-k)\times k}(q)| C J_{\obmuk} =C\}\big|=q^{r(n-k)}$.
Therefore the equation \eqref{eq:ccmu-n-decom} is compatible with the decomposition into four terms in \eqref{eq:ccmu-n}.

\end{rem}

\section{Stability of the centers $\ZZ_n(q)$}
 \label{sec:stable}

 In this section, we examine the interrelations among reflection lengths, fixed point subspaces, and modified types. We show that the associated graded algebra $\mathscr{G}_n(q)$ to the centers $\ZZ_n(q)$ as filtered algebras with respect to the reflection length has structure constants independent of $n$. This leads to a formulation of a stable center, which governs the algebras $\mathscr{G}_n(q)$ for all $n$.

\subsection{The reflection length and modified type}

Recall $V_n=\Fq^n$. For $g\in G_n$, the fixed point subspace by $g$ is denoted by
\[
V_n^g:=\ker(g-1)=\{v\in V_n |gv=v\}.
\]
An element $s$ in $G_n$ is a {\em reflection} if its fixed point subspace has codimension 1. Let $R_n$ be the set of reflections in $G_n$. Then $R_n$ is a generating set for the group $G_n$, since all of the elementary matrices used in Gaussian elimination are reflections and every invertible matrix is row equivalent to the identity matrix. The {\em reflection length} of an element $g\in G_n$ is defined by
\begin{equation}
  \label{eq:length}
\ell(g) := \min \big\{k \big| g = r_1r_2\cdots r_k \text{ for some } r_i\in R_n \big\}.
\end{equation}
The combinatorics of partial orders on $G_n$ arising from the reflection lengths has been studied in \cite{HLR17}.
Recall the codimension ${\rm codim} V_n^g  =n -\dim V_n^g={\rm rank}(g-I_n)$. The reflection length has the following simple and useful geometric interpretation.

\begin{lem}
\cite[Propositions~ 2.9, 2.16]{HLR17}\label{lem:length}
  \label{lem:l=cod}
\begin{enumerate}
\item
For $g\in G_n$, we have $\ell(g) ={\rm codim} V_n^g$.
\item
Suppose $g,h\in G_n$. Then $\ell(gh)\leq\ell(g)+\ell(h)$.
\item
If $\ell(gh)=\ell(g)+\ell(h)$, then $V_n^g\cap V_n^h=V_n^{gh}$ and $V_n=V_n^g+V_n^h$.
\end{enumerate}
\end{lem}

For $g\in G_n\subset G_{n+m}$, the fixed subspaces satisfy $V^g_{n+m}=V^g_n\oplus V'_{m}$.
Hence by Lemma~ \ref{lem:length}, the length function is compatible with the embedding $G_n\subset G_{n+m}$.
In particular, a reflection $s$ in $G_n$ is also a reflection in $G_{n+m}$, that is, $R_n\subset R_{n+m}$, and hence
accordingly the set of reflections in $G_\infty$ is the union
\[
R:=\cup_{n\geq 1}R_n.
\]
Then $g\in G_\infty$ has the length $\ell(g)=\min\{k|g=r_1r_2\cdots r_k, \text{ for some }r_i\in R\}$.
It follows readily by \eqref{eq:Jbla} that a reflection $s$ is similar to the canonical form
$$
  \begin{bmatrix}
   1 & {\bf 1}&  0 & \cdots & 0 &0\\
   0 & 1 & 0 & \cdots & 0 &0\\
   0 & 0 & 1 &\cdots  & 0 &0\\
   \vdots & \vdots & \vdots & \vdots & \vdots & \vdots\\
   0 & 0& 0 & \cdots &  1 & 0\\
   0 & 0& 0 & \cdots & 0  & 1
  \end{bmatrix}
\qquad
\text{or }
\qquad
   \begin{bmatrix}
   \xi & 0&  0 & \cdots & 0 &0\\
   0 & 1 & 0 & \cdots & 0 &0\\
   0 & 0 & 1 &\cdots  & 0 &0\\
   \vdots & \vdots & \vdots & \vdots & \vdots & \vdots\\
   0 & 0& 0 & \cdots &  1 & 0\\
   0 & 0& 0 & \cdots & 0  & 1
  \end{bmatrix}\quad (\text{for } \xi\in\Fq \backslash \{0,1\}).
$$
Equivalently, an element in $G_\infty$ is a reflection if and only if its modified type $\bmu$ satisfies $\|\bmu\|=1$.

\begin{lem} \label{lem:length-ineq}
$\quad$
\begin{enumerate}
\item
If $g\in\mathscr{K}_\bmu$, then $\ell(g)=\|\bmu\|.$

\item
If the modified types of $g,h$, $gh \in G_\infty$ are $\bla,\bmu$ and $\bnu$, then $\|\bnu\|\leq \|\bla\|+\|\bmu\|$.
\end{enumerate}
\end{lem}

\begin{proof}
As the reflection length is conjugation invariant, we can take $g=J_\obmun$; cf.  \eqref{eq:Jbla} and \eqref{eq:obmu}-\eqref{eq:bmuarr} for notations. Following the definitions, one checks directly that ${\rm codim} V_n^{g} =\|\bmu\|$. Hence (1) follows by using Lemma \ref{lem:length}. Part~(2) follows by (1) and Lemma~\ref{lem:l=cod}(2).
\end{proof}

\subsection{A normal form of triples}

The following proposition is a crucial step in the proof of the stability as formulated in the subsequent subsections.

\begin{prop} \label{prop:length-eq}
Let $g,h$, $gh \in G_n$ be of modified type $\bla,\bmu$ and $\bnu$, respectively.
Suppose $\|\bnu\|= \|\bla\|+\|\bmu\|$.
Set $k =\|\bnu\|+\ell(\bnu^e)$.
Then there exists $z\in G_n$ and $\bg,\bh\in G_{k}$ such that
\[
zgz^{-1}=
  \begin{bmatrix}
   \bg & 0\\
   0&I_{n-k}
  \end{bmatrix},\qquad
zhz^{-1}=
  \begin{bmatrix}
   \bh & 0\\
   0&I_{n-k}
  \end{bmatrix},\qquad
zghz^{-1}=
  \begin{bmatrix}
   \bg\bh & 0\\
   0&I_{n-k}
  \end{bmatrix}.
\]
\end{prop}
Note that $\bg,\bh\in G_{k}$ above have modified types $\bla, \bmu$ respectively, and $\bg\bh$ is of type $\obnu$. We regard the triple of matrices in the proposition above as a {\em normal form} for the triple $(g,h,gh)$ which satisfies $\ell(g) +\ell(h) =\ell(gh)$.

\begin{proof}
Since the modified type of $gh$ is $\bnu$,  the type of $gh$ is $\obnun$.
Then by \eqref{eq:Jbla}, $gh$ is conjugate to $J_{\obnun}$, and thus there exists a basis $\{v_1,\ldots,v_k,v_{k+1},\ldots,v_n\}$ of $V_n=\Fq^n$ such that $(gh)v_i\in\text{span-}\{v_1,\ldots,v_k\}$ for $1\leq i\leq k$ and $(gh)v_i=v_i$ for $k+1\leq i\leq n$.
By Lemma~ \ref{lem:length-ineq}, we have
\begin{equation}
  \label{eq:3gh}
 \ell(gh)=\|\bnu\|=\|\bla\|+\|\bmu\|=\ell(g)+\ell(h),
\end{equation}
and then by Lemma \ref{lem:l=cod}, $gv_i=v_i$ and $hv_i=v_i$ for $k+1\leq i\leq n$. Therefore, there exist elements $z\in G_n$ and $\bg,\bh\in G_k$
and $g',h'\in M_{(n-k)\times k}(q)$ such that
\[
zgz^{-1}=
  \begin{bmatrix}
   \bg & 0\\
   g'&I_{n-k}
  \end{bmatrix}
,\qquad
zhz^{-1}=
  \begin{bmatrix}
   \bh & 0\\
   h'&I_{n-k}
  \end{bmatrix},\qquad
zghz^{-1}=
  \begin{bmatrix}
   \bg\bh & 0\\
   0&I_{n-k}
  \end{bmatrix}.
\]

It remains to show $g'=0=h'$. It follows by $zghz^{-1}=(zgz^{-1})(zhz^{-1})$ that
\begin{equation}\label{eq:g2h1}
g'\bh+h'=0.
\end{equation}
Furthermore, the following holds:
\begin{align*}
\ell(gh)=\ell(zghz^{-1})=\ell(\bg \bh)&\leq \ell(\bg)+\ell(\bh)\\
&=\text{rank}(\bg-I_k)+\text{rank}(\bh-I_k)\\
&\leq \text{rank}
  \begin{bmatrix}
   \bg-I_k & 0\\
   g'&0
  \end{bmatrix}
+\text{rank}
  \begin{bmatrix}
   \bh-I_k & 0\\
   h'&0
  \end{bmatrix}
\\
&=\text{rank}(zgz^{-1}-I_n)+\text{rank}(zhz^{-1}-I_n)\\
&=\ell(g)+\ell(h).
\end{align*}
From this together with the equality $\ell(gh)=\ell(g)+\ell(h)$ from \eqref{eq:3gh} we conclude that all the inequalities above are indeed equalities:
\begin{align}
\ell(\bg\bh)&=\ell(\bg)+\ell(\bh),
\label{eq:l-g1h1}\\
\text{rank}(\bg-I_k)&=\text{rank}
  \begin{bmatrix}
   \bg-I_k & 0\\
   g'&0
  \end{bmatrix},
\label{eq:rankg-1}
\\
\text{rank}(\bh-I_k)&=\text{rank}
  \begin{bmatrix}
   \bh-I_k & 0\\
   h'&0
  \end{bmatrix}.\label{eq:rankh-1}
\end{align}
Then by \eqref{eq:rankg-1} and \eqref{eq:rankh-1}, there exist $A,B\in M_{(n-k)\times k} (q)$ such that
\begin{equation}
  \label{eq:ghAB}
g'=A(\bg-I_k), \qquad h'=B(\bh-I_k).
\end{equation}

By \eqref{eq:l-g1h1} and the invertibility of $\bh$, we have
\begin{align*}
\text{rank}(\bg\bh-I_k)=\text{rank}(\bg-I_k)+\text{rank}(\bh-I_k)=\text{rank} \big( (\bg-I_k)\bh \big)+\text{rank}(\bh-I_k),
\end{align*}
which can then be rewritten as
\begin{equation}\label{eq:rank-g1h1}
\text{rank} \big((\bg-I_k)\bh+(\bh-I_k) \big)=\text{rank} \big( (\bg-I_k)\bh \big) +\text{rank}(\bh-I_k).
\end{equation}
Let $U_1$ and $U_2$ be the subspaces of $V_k$ spanned by the row vectors of $(\bg-I_k)\bh$ and $\bh-I_k$, respectively.
Then by \eqref{eq:rank-g1h1} we have
\begin{equation}\label{eq:U1U2}
U_1\cap U_2=0.
\end{equation}
On the other hand, by \eqref{eq:g2h1} and \eqref{eq:ghAB}, we have
\[
A(\bg-I_k)\bh+B(\bh-I_k)=0.
\]
Observe that the row vectors of $A(\bg-I_k)\bh$ belong to $U_1$ while the row vectors of $\bh-I_k$ belong to $U_2$.
Then by \eqref{eq:U1U2} we obtain
\[
A(\bg-I_k)\bh=0,\qquad B(\bh-I_k)=0.
\]
Hence we have $h'=B(\bh-I_k)=0$, and then by \eqref{eq:g2h1}, $g'= -h' \bh^{-1} =0$.

The proposition is proved.
\end{proof}

\subsection{A stability property}

Recall from Lemma~\ref{lem:Kbmu} that the set $\{K_\bmu(n) \neq 0 | \bmu \in \mathcal{P}(\Phi)\}$  forms the class sum $\mathbb{Z}$-basis for the center $\ZZ_n(q)$.
Given $\bla,\bmu\in\mP(\Phi)$, we can write the product in  $\ZZ_n(q)$ as in \eqref{eq:a(n)},
where the structure constants $a^\bnu_{\bla\bmu}(n)$ are zero unless $\|\bnu\|\leq \|\bla\|+\|\bmu\|$
by Lemma~ \ref{lem:length-ineq}. For $\bnu$ with $\|\bnu\|+\ell(\bnu^e)\leq n$, the coefficient $a^\bnu_{\bla\bmu}(n)$
is uniquely determined.

\begin{thm}   \label{thm:indep}
Let $\bla,\bmu,\bnu\in\mP(\Phi)$.
If $\|\bnu\|= \|\bla\|+\|\bmu\|$,
then $a^\bnu_{\bla\bmu}(n)$ is a nonnegative integer independent of $n$. 
{\em (In this case, we shall write $a^\bnu_{\bla\bmu}(n)$ as $a^\bnu_{\bla\bmu}$.)}
\end{thm}

\begin{proof}
Set  $k=\|\bnu\|+\ell(\bnu^e)$.

Observe that $G_\infty$ acts on the set of pairs
$$
\mathcal{T}=\{(g,h)\in G_\infty\times G_\infty| g\in\mathscr{K}_\bla, h\in\mathscr{K}_\bmu, gh\in\mathscr{K}_\bnu\}
$$
by simultaneous conjugation: $z. (g,h) =(zgz^{-1}, zhz^{-1})$. We say $(g,h)$ and $(zgz^{-1}, zhz^{-1})$ are conjugate, and so the set $\mathcal{T}$ is a union of such conjugate classes. We claim that the set of conjugate classes of such pairs in $\mathcal{T}$ is finite. Indeed, by Proposition \ref{prop:length-eq} each pair$(g,h)\in\mathcal{T}$ is conjugate to some pair lying in $G_k\times G_k$, which is a finite set. Denote the conjugate classes in $\mathcal{T}$ by $\mathscr{C}_1,\ldots,\mathscr{C}_r$.

Suppose $n\geq \|\bnu\|+\ell(\bnu^e)$. Then by Proposition \ref{prop:length-eq}, there exists $(g_i,h_i)\in \mathscr{C}_i\cap (G_n\times G_n)$ for each $1\leq i\leq r$. Moreover, $\mathscr{C}_i \cap (G_n\times G_n)$ is a single conjugate class in $\mathcal{T} \cap (G_n\times G_n)$ under the simultaneous conjugation of $G_n$.  By applying Proposition ~\ref{prop:length-eq} once more, there exist $\overline{g}_i,\overline{h}_i\in G_k$ such that
the pair $(g_i,h_i)$ is conjugate to the pair $(\widetilde{g}_i, \widetilde{h}_i)$, where we have denoted
\[
\widetilde{g}_i=\begin{bmatrix} \overline{g}_i& 0\\ 0& I_{n-k}\end{bmatrix},
\qquad
\widetilde{h}_i=\begin{bmatrix} \overline{h}_i& 0\\ 0& I_{n-k}\end{bmatrix}.
\]
Let us denote the stabilizer of the pair $(\widetilde{g}_i, \widetilde{h}_i)$ by $G_{n;(\widetilde{g}_i, \widetilde{h}_i)}$. Then
for each $i$, we have
\begin{align}
\big| G_{n;(\widetilde{g}_i, \widetilde{h}_i)} \big|
=&\Big|\big\{z\in G_n\big| z\widetilde{g}_iz^{-1}=\widetilde{g}_i,
z\widetilde{h}_iz^{-1}=\widetilde{h}_i\big\}\Big|
\notag\\
\stackrel{\rm (i)}{=} &\Bigg|\Big\{z=\begin{bmatrix} z_1& z_2\\ z_3& z_4\end{bmatrix}\in M_{n\times n}(q)\Big| z_1\in G_k, z_4\in G_{n-k}, z\widetilde{g}_iz^{-1}=\widetilde{g}_i, z\widetilde{h}_iz^{-1}=\widetilde{h}_i\Big\}\Bigg|
\notag\\
=&
\Big|\big\{z_1\in G_k \big|  z_1\overline{g}_iz_1^{-1}=\overline{g}_i,z_1\overline{h}_iz_1^{-1}=\overline{h}_i\big\}\Big|
\cdot \Big|G_{n-k} \Big|\cdot
  \notag\\
\cdot&
\big|\big\{z_2\in M_{k\times(n-k)}(q)\big| \overline{g}_iz_2=z_2=\overline{h}_iz_2\big\}\big|
\cdot\big|\big\{z_3\in M_{(n-k)\times k}(q)\big| z_3\overline{g}_i=z_3=z_3\overline{h}_i\big\}\big|
\notag\\
\stackrel{\rm (ii)}{=} &\Big|\big\{z_1\in G_k \big| z_1\overline{g}_iz_1^{-1}=\overline{g}_i,z_1\overline{h}_iz_1^{-1}=\overline{h}_i\big\}\Big|
\cdot \Big|G_{n-k} \Big|\cdot
 \notag\\
\cdot& \Big|\big\{z_2\in M_{k\times(n-k)}(q)\big|
\overline{g}_i\overline{h}_iz_2=z_2\big\}\Big|
\cdot\Big|\big\{z_3\in M_{(n-k)\times k}(q)\big|
z_3\overline{g}_i\overline{h}_i=z_3\big\}\Big|, \notag
\end{align}
where the equality (i) follows from Corollary \ref{cor:invertible} since the type of $\overline{g}_i\overline{h}_i$
is $\obnu$ and the equality (ii) follows from the following.

\vspace{2mm}
\noindent {\bf Claim.}
\begin{enumerate}
\item
For $z_2\in M_{k\times(n-k)}(q)$, then $\overline{g}_iz_2=z_2=\overline{h}_iz_2$ if and only if $\overline{g}_i\overline{h}_iz_2=z_2$.
\item
For $z_3\in M_{(n-k)\times k}(q)$, then $z_3\overline{g}_i=z_3=z_3\overline{h}_i$ if and only if $z_3\overline{g}_i\overline{h}_i=z_3$.
\end{enumerate}

We prove the Claim. Denote by $C_2$ an arbitrary column vector of $z_2$.
Since $\ell(\overline{g}_i\overline{h}_i)=\|\bnu\|=\|\bla\|+\|\bmu\|=\ell(\overline{g}_i)+\ell(\overline{h}_i)$, applying  Lemma \ref{lem:l=cod} we obtain
$C_2 \in V_k^{\overline{g}_i} \cap V_k^{\overline{h}_i}$ if and only if $C_2 \in V_k^{\overline{g}_i \overline{h}_i}$, i.e.,
$\overline{g}_iC_2=C_2=\overline{h}_iC_2$ if and only if $\overline{g}_i \overline{h}_i C_2=C_2$, whence (1).
Noting the reflection length is transpose invariant, we have $\ell(\overline{h}_i^\intercal\overline{g}_i^\intercal)= \ell(\overline{h}_i^\intercal)+\ell(\overline{g}_i^\intercal)$. Then Claim (2) follows by (1).
This completes the proof of the Claim.

Using the above identity for $\big| G_{n;(\widetilde{g}_i, \widetilde{h}_i)} \big|$ and Corollary \ref{cor:Kmu-n} we obtain
\begin{align*}
\ds a^{\bnu}_{\bla\bmu}(n)
&=\sum^r_{i=1}\frac{\big|\mathscr{C}_i\cap (G_n\times G_n)\big|}{\big|\mathscr{K}_\bnu\cap G_n\big|}
=\sum^r_{i=1}\frac{|\mathcal{A}_{\obnun}| }{\big| G_{n;(\widetilde{g}_i, \widetilde{h}_i)} \big|}\\
&=\sum^r_{i=1}\frac{|\mathcal{A}_{\obnu}|}{\Big|\big\{z_1\in G_k\big| z_1\overline{g}_iz_1^{-1}=\overline{g}_i,z_1\overline{h}_iz_1^{-1}=\overline{h}_i\big\}\Big|},
\end{align*}
which is independent of $n$. The theorem is proved.
\end{proof}

\subsection{The stable center}

Let $\mathcal{K}_m$ be the subspace of $\ZZ_n(q)$ spanned by the elements $K_{\bla}(n)$ with $\|\bla\|\leq m$ and $\bla\in\mP(\Phi)$. Thanks to Lemma~ \ref{lem:length-ineq}, the assignment of degree $\|\bla\|$ to $K_\bla(n)$
provides $\ZZ_n(q)$ a filtered ring structure with the filtration $0\subset\mathcal{K}_0\subset \mathcal{K}_1\subset\mathcal{K}_2\subset\cdots\subset \ZZ_n(q)$.  Then we can define the associated graded algebra denoted by $\mathscr{G}_n(q)$ as follows. As a vector space $\mathscr{G}_n(q)=\oplus_{i\geq 0}(\mathcal{K}_{i}/\mathcal{K}_{i-1})$ where we set $\mathcal{K}_{-1}=0$ and the multiplication satisfies $(x+\mathcal{K}_{i-1})(y+\mathcal{K}_{j-1})=xy+\mathcal{K}_{i+j-1}$ for $x\in \mathcal{K}_i, y\in\mathcal{K}_j$ and $i,j\geq 0$.
Meanwhile, introduce a graded associative $\mathbb{Z}$-algebra $\mathscr{G}(q)$ with a basis given by the symbols $K_\bla$ indexed by $\bla\in\mP(\Phi)$,
and with multiplication
given by
\begin{equation}
  \label{eq:KKaK}
K_\bla K_\bmu=\sum_{\|\bnu\|=\|\bla\|+\|\bmu\|}a^\bnu_{\bla\bmu}K_\bnu.
\end{equation}
Note $K_\emptyset$ is the unit of $\mathscr{G}(q)$.
The following summarizes the above discussions.

\begin{thm} \label{thm:stable}
The graded $\mathbb{Z}$-algebra $\mathscr{G}_n(q)$ has the multiplication given by
$$
K_\bla(n) K_\bmu(n)=\sum_{\|\bnu\|=\|\bla\|+\|\bmu\|}a^\bnu_{\bla\bmu}K_\bnu(n),
$$
for $\bla,\bmu\in\mP(\Phi)$. 
Moreover, we have a surjective algebra homomorphism $\mathscr{G}(q) \twoheadrightarrow\mathscr{G}_n(q)$
for each $n$, which maps $K_\bla$ to $K_\bla(n)$ for all $\bla\in\mP(\Phi)$.
\end{thm}
We will refer to $\mathscr{G}(q)$ as the {\em stable center} associated to the family of finite general linear groups.
This algebra can be viewed as the inverse limit of the projective system of algebras $\{\mathscr{G}_n(q)\}_{n\ge 1}$.

\begin{rem}\label{rem:Meliot}
Theorems~\ref{thm:indep} and \ref{thm:stable} are the finite general linear group counterparts of
analogous results for symmetric groups in \cite{FH59} and for wreath products in \cite{W04}.
It is shown by M\'eliot \cite{M14} that the structure constants $a_{\bla\bmu}^\bnu (n)$ for the center $\ZZ_n$ defined in \eqref{eq:a(n)} are polynomials in $q^n$, an analogue of another theorem of Farahat-Higmann for symmetric groups \cite{FH59}.
However, we have difficulties in verifying some key details in M\'eliot's approach such as the equivalence between the assertions 2 and 3 in \cite[Definition 2.3]{M14} as well as the proof of \cite[Lemma 2.21]{M14}.
\end{rem}

\section{Computations in the stable center}
 \label{sec:computation}

In this section, we compute various structure constants $a_{\bla\bmu}^\bnu$ in Theorems~\ref{thm:indep} and \ref{thm:stable}. 

\subsection{Multiplication of class sums of reflections }

For $r\ge 1$ and $f\in\Phi$, we define the {\em single cycles} $(r)_f \in \mP(\Phi)$
by letting $(r)_f(f)=(r)$ and $(r)_f(f')=\emptyset$ for $f'\neq f$.
Call $(r)_f$ a $r$-cycle of degree $d(f)$.
Denote by $\Fq^*= \Fq\backslash \{0\}$.

We shall compute the structure constants arising in the product of class sums of reflections in $G_n$. We proceed in three separate cases,
depending on the number of unipotent class sums involved in the multiplication; see Lemmas~\ref{lem:xi=1}--\ref{lem:xi-not-1} below.
\begin{lem}\label{lem:xi=1}
Suppose $\xi',\eta'\in\Fq^*$ and $\eta \in \Fq \backslash \{0,1\}$.
Let $\bla=(1)_{t-1},\bmu=(1)_{t-\eta}$ and $f=t^2+a_2t+a_1\in\Phi$.
Then
\begin{align*}
a^{(1)_{t-\xi'}\cup(1)_{t-\eta'}}_{\bla\bmu}&=\left\{
\begin{array}{cc}
q-1, & \text{ if }\xi'\eta'=\eta, \xi'\neq\eta', \\
2q-1, &\text{ if } \{\xi',\eta'\}=\{1,\eta\},\\
0,& \text{ otherwise},
\end{array}
\right.\\
a^{(2)_{t-\xi'}}_{\bla\bmu}&=\left\{
\begin{array}{cc}
q,&\text{if } \xi'^2=\eta,\\
0,& \text{ otherwise},
\end{array}
\right.\\
a^{(1)_f}_{\bla\bmu}&=\left\{
\begin{array}{cc}
q+1, & \text{ if }a_1=\eta,\\
0,& \text{otherwise}.
\end{array}
\right.
\end{align*}
\end{lem}

\begin{proof}
We separate the proof into several cases.

(1). Let us compute $a^{(1)_{t-\xi'}\cup(1)_{t-\eta'}}_{\bla\bmu}$.
Clearly $a^{(1)_{t-\xi'}\cup(1)_{t-\eta'}}_{\bla\bmu}=0$ if $\xi'\eta'\neq \eta$.
Assume $\xi'\eta'=\eta$. We first consider the case $\{\xi',\eta'\}\neq \{1,\eta\}$.
Let
\[
\Gamma_1=\bigg\{A~\Big|~ A\sim\begin{bmatrix} \eta^{-1} &0\\ 0&1\end{bmatrix},
\quad
 \begin{bmatrix} \xi'&0\\ 0&\eta'\end{bmatrix}A\sim \begin{bmatrix} 1 &1\\0&1\end{bmatrix}\bigg\}.
\]
Then $a^{(1)_{t-\xi'}\cup(1)_{t-\eta'}}_{\bla\bmu}=|\Gamma_1|$.
Observe that $A=\begin{bmatrix} a_{11} & a_{12}\\ a_{21} &a_{22}\end{bmatrix}$ belongs to $\Gamma_1$
if and only if the following holds:
\begin{align}
      \label{eq:Gamma-1}
\begin{split}
a_{11}+a_{22}  =\eta^{-1}+1, &\qquad a_{11}a_{22}-a_{12}a_{21}=\eta^{-1},
\\
 \xi'a_{11}+\eta'a_{22} =2, & \qquad a_{12}\neq 0 \; \text{ or } \; a_{21}\neq 0.
\end{split}
\end{align}
Hence if $\xi'=\eta'$ then $a^{(1^2)_{t-\xi'}}_{\bla\bmu}=|\Gamma_1| =0$.
Assume $\xi'\neq \eta'$.
Since $\xi',\eta'\neq 1$, a direct calculation shows that \eqref{eq:Gamma-1} is equivalent to
\begin{align*}
a_{11}&=\frac{2-\eta'(\eta^{-1}+1)}{\xi'-\eta'},\quad a_{22}=\frac{\xi'(\eta^{-1}+1)-2}{\xi'-\eta'},  \\
a_{12}a_{21}&=\frac{-\eta^{-1}((\xi'+\eta')-(\eta+1))^2}{(\xi'-\eta')^2}.
\end{align*}
Hence $a^{(1)_{t-\xi'}\cup(1)_{t-\eta'}}_{\bla\bmu}=|\Gamma_1|=q-1$ since $\xi'+\eta'\neq\eta+1$.

In the situation that $\{\xi',\eta'\} =\{1,\eta\}$, it is a special case of Proposition \ref{prop:a-union}
in the subsequent section, which includes a detailed proof.

\vspace{2mm}
(2). We now compute $a^{(2)_{t-\xi'}}_{\bla\bmu}$. Note $a^{(2)_{t-\xi'}}_{\bla\bmu}=0$ unless $(\xi')^2=\eta$. Assume $(\xi')^2=\eta$ and then clearly $\xi'\neq 1$.
Set
\[
\Gamma_3=\Big\{A ~\Big|~ A\sim\begin{bmatrix} \eta^{-1} &0\\ 0&1\end{bmatrix}, \begin{bmatrix} \xi'&1\\ 0&\xi'\end{bmatrix}A\sim \begin{bmatrix} 1 &1\\0&1\end{bmatrix}\Big\}.
\]
Then $a^{(2)_{t-\xi'}}_{\bla\bmu}=|\Gamma_3|$, and $A=\begin{bmatrix} a_{11} & a_{12}\\ a_{21} & a_{22}\end{bmatrix}$ belongs to $\Gamma_3$ if and only if the following holds:
\begin{align}
  \label{eq:Gamma-3}
  \begin{split}
a_{11}+a_{22} &=\eta^{-1}+1,\quad a_{11}a_{22}-a_{12}a_{21}=\eta^{-1}, \\
 \xi'a_{11}+\xi'a_{22}+a_{21} &=2,\quad
a_{21}\neq 0 \; \text{ or } \; \xi'a_{12}+a_{22}\neq 0.
\end{split}
\end{align}
Using $(\xi')^2=\eta$, a direct calculation shows that \eqref{eq:Gamma-3} is equivalent to
$$
a_{22}=(\eta^{-1}+1)-a_{11},\quad a_{21}=2-\xi'(\eta^{-1}+1),\quad a_{12}=\frac{a_{11} \big((\eta^{-1}+1)-a_{11} \big)-\eta^{-1}}{2-\xi'(\eta^{-1}+1)}.
$$
Hence the number of solutions to \eqref{eq:Gamma-3}
is $q$, which implies $a^{(2)_{t-\xi'}}_{\bla\bmu}=q$.

\vspace{2mm}
(3). Suppose $f(t)=t^2+a_2t+a_1\in\Phi$.
Let $\bnu=(1)_f$ and
\[
\Gamma_4=\bigg\{B~\Big|~ B\sim\begin{bmatrix} \eta &0\\ 0&1\end{bmatrix},
\quad
\det\Big(tI_2-\begin{bmatrix} 1&1\\ 0&1\end{bmatrix}B\Big)=f(t) \bigg\}.
\]
Since $f(t)$ is irreducible, we observe that a matrix $C\in G_2$ is conjugate to $J(f)$ if and only if the characteristic polynomial of $C$ is $f(t)$, and so
\[
\Gamma_4=\bigg\{B~\Big|~ B\sim\begin{bmatrix} \eta &0\\ 0&1\end{bmatrix},
\quad
\begin{bmatrix} 1&1\\ 0&1\end{bmatrix}B\sim J(f)\bigg\}.
\]
Therefore we have
\begin{equation}\label{eq:a1f-2}
a^{(1)_f}_{\bla\bmu}=\frac{|\mathscr{K}_\bla\cap G_2|\cdot|\Gamma_4|}{|\mathscr{K}_\bnu\cap G_2|}
=\frac{(q^2-1)\cdot|\Gamma_4|}{q^2-q}.
\end{equation}

Clearly $B=\begin{bmatrix} b_{11} & b_{12}\\ b_{21} &b_{22}\end{bmatrix}$ belongs to $\Gamma_4$
if and only if the following holds:
\begin{equation}\label{eq:Gamma-4}
b_{11}+b_{22}=\eta+1,\quad b_{11}b_{22}-b_{12}b_{21}=\eta, \quad  b_{11}+b_{21}+ b_{22}=-a_2, \quad b_{11}b_{22}-b_{12}b_{21}=a_1.
\end{equation}
So $a^{(1)_f}_{\bla\bmu}=0$ unless $a_1=\eta$.
Now suppose $a_1=\eta$.
Since $f(t)=t^2+a_2t+\eta$ is irreducible, we have $a_2+(\eta+1)\neq 0$, and hence \eqref{eq:Gamma-4} is equivalent
\[
b_{22}=(\eta+1)-b_{11},\quad b_{21}=-a_2-(\eta+1),\quad b_{12}=\frac{(b_{11}-1)(b_{11}-\eta)}{a_2+(\eta+1)}.
\]
This implies $|\Gamma_4|=q$. Hence by \eqref{eq:a1f-2} we obtain $a^{(1)_f}_{\bla\bmu}=q+1$.

The lemma is proved.
\end{proof}
\begin{lem}\label{lem:xi=eta=1}
Let $\bla=(1)_{t-1}$.
Suppose $\xi,\eta\in\Fq^*$ and $f=t^2+a_2t+a_1\in\Phi$.
Then
\begin{align*}
a^{(1)_{t-\xi}\cup(1)_{t-\eta}}_{\bla\bla}&=\left\{
\begin{array}{cc}
q-1, & \text{ if }\xi\eta=1, \xi\neq 1, \\
q^2+q,&\text{ if }\xi=\eta=1,\\
0,& \text{otherwise},
\end{array}
\right.\\
a^{(2)_{t-\xi}}_{\bla\bla}&=\left\{
\begin{array}{cc}
q,&\text{ if } \xi^2=1,\xi\neq 1,\\
2q, &\text{ if }\xi=1,\\
0,& \text{otherwise},
\end{array}
\right.\\
a^{(1)_f}_{\bla\bla}&=\left\{
\begin{array}{cc}
q+1, & \text{ if }a_1=1,\\
0,& \text{otherwise}.
\end{array}
\right.
\end{align*}
\end{lem}
\begin{proof}
The methods and calculations used in the proof of Lemma ~\ref{lem:xi=1} can also be applied to prove the formulas in the case $\xi=\eta=1$.
To compute $a^{(1^2)_{t-1}}_{\bla\bmu}$,
let
$$
\Gamma_5=\left\{B~\bigg|~ B\sim \begin{bmatrix}1&1&0&0\\0&1&0&0\\0&0&1&0\\0&0&0&1 \end{bmatrix},
\quad
\begin{bmatrix}1&1&0&0\\0&1&0&0\\0&0&1&1\\0&0&0&1 \end{bmatrix}B\sim \begin{bmatrix}1&1&0&0\\0&1&0&0\\0&0&1&0\\0&0&0&1 \end{bmatrix}\right\}.
$$
Write $B=\begin{bmatrix}b_{11}& b_{12}& b_{13} & b_{14} \\ b_{21}& b_{22}& b_{23} & b_{24} \\ b_{31}& b_{32}& b_{33}& b_{34}\\ b_{41} &b_{42}&b_{43}&b_{44}\end{bmatrix}\in G_4$ and
$\widetilde{B}= \begin{bmatrix}1&1&0&0\\0&1&0&0\\0&0&1&1\\0&0&0&1 \end{bmatrix}\cdot B$.
Observe that $B$ belongs to $\Gamma_5$
if and only if the following holds:
\begin{equation}\label{eq:Gamma-5}
b_{11}+b_{22}+b_{33}+b_{44}=4,\quad {\rm rank} (B-I_4)=1,\quad
{\rm rank} (\widetilde{B}-I_4)=1, \quad b_{21}+b_{43}=0.
\end{equation}
A direct calculation shows that \eqref{eq:Gamma-5} is equivalent to
\begin{align*}
b_{11}=b_{22}&=b_{33}=b_{44}=1,\quad b_{13}=b_{21}=b_{23}=b_{24}=b_{31}=b_{41}=b_{42}=b_{43}=0, \\
b_{14}b_{32}&=b_{12}(q-1-b_{12}),\quad b_{34}=q-1-b_{12}.
\end{align*}
This implies $a^{(1^2)_{t-1}}_{\bla\bmu}=|\Gamma_5|=2(2q-1)+(q-2)(q-1)=q^2+q$.

We omit the detailed proofs for the other cases.
\end{proof}

The follow formulas can be proved by using arguments similar to the proof of Lemma \ref{lem:xi=1}. We omit the details which can be found
in arXiv verion 1 of this paper. This coincides with the computation in \cite[Theorem 4.1]{M14}; see however Remark \ref{rem:Meliot}.

\begin{lem}  
 \label{lem:xi-not-1}
Let $\xi,\eta\in\Fq\setminus\{0,1\}$ and $\xi',\eta'\in\Fq^*$.
Let $\bla=(1)_{t-\xi},\bmu=(1)_{t-\eta}$ and $f=t^2+a_2t+a_1\in\Phi$.
Then
\begin{align*}
a^{(1)_{t-\xi'}\cup(1)_{t-\eta'}}_{\bla\bmu}&=\left\{
\begin{array}{cc}
q-1, & \text{ if }\xi'\eta'=\xi\eta, \xi'\neq\eta', \{\xi',\eta'\}\neq\{\xi,\eta\},\\
2q-1,& \text{ if }\xi'\eta'=\xi\eta, \xi'\neq\eta', \{\xi',\eta'\}=\{\xi,\eta\},\\
q^2+q,&\text{ if }\xi'=\eta'=\xi=\eta,\\
0,& \text{ otherwise},
\end{array}
\right.\\
a^{(2)_{t-\xi'}}_{\bla\bmu}&=\left\{
\begin{array}{cc}
q,&\text{if } \xi'^2=\xi\eta, \xi'\notin\{\xi,\eta\}\\
2q, & \text{ if }\xi'=\xi=\eta,\\
0,& \text{ otherwise},
\end{array}
\right.\\
a^{(1)_f}_{\bla\bmu}&=\left\{
\begin{array}{cc}
q+1, & \text{ if }a_1=\xi\eta,\\
0,& \text{ if }a_1\neq \xi\eta.
\end{array}
\right.
\end{align*}
\end{lem}

While the computations to derive the formulas in Lemmas~ \ref{lem:xi=1}-\ref{lem:xi-not-1} have to be carried out separately and the sizes of the matrices involved are different, we find it rather remarkable that these formulas afford a uniform reformulation. This is summarized in the following theorem.

\begin{thm}
 \label{thm:2reflections}
Suppose $\xi,\eta,\xi',\eta'\in\Fq^*$ and $f\in\Phi$ with $\deg f=2$.
Let $\bla=(1)_{t-\xi},\bmu=(1)_{t-\eta}$.
Then $a^\bnu_{\bla\bmu}=0$ if $\det J_{\obnu}\neq \det J_{\obla}\cdot\det J_{\obmu}$ for $\bnu\in\mP(\Phi)$ with $\|\bnu\|=2$.
Otherwise, we have the following complete list:
\begin{align*}
a^{(1)_{t-\xi'}\cup(1)_{t-\eta'}}_{\bla\bmu}&=\left\{
\begin{array}{cc}
q-1, & \text{ if }\xi'\neq\eta', \{\xi',\eta'\}\neq\{\xi,\eta\}, \\
2q-1, & \text{ if }\xi'\neq\eta', \{\xi',\eta'\}=\{\xi,\eta\}, \\
q^2+q,&\text{ if }\xi'=\eta'=\xi=\eta,\\
0,& \text{ otherwise},
\end{array}
\right.\\
a^{(2)_{t-\xi'}}_{\bla\bmu}&=\left\{
\begin{array}{cc}
q,&\text{if } \xi'\notin\{\xi,\eta\},\\
2q,&\text{ if }\xi'\in\{\xi,\eta\},\\
0,& \text{otherwise},
\end{array}
\right.\\
a^{(1)_f}_{\bla\bmu}&=q+1.
\end{align*}
\end{thm}

\subsection{Computation on $a^{\bla\cup\bmu}_{\bla\bmu}$}

In the setting of symmetric groups \cite{FH59} and wreath products \cite{W04}, the structure constants $a^{\bla\cup\bmu}_{\bla\bmu}$ are among the easiest to compute.  In our setting, these structure constants are not as straightforward to compute in general. We shall present $a^{\bla\cup\bmu}_{\bla\bmu}$ in some simplest nontrivial cases.

\subsubsection{}
We first compute some cases when $\bla$ is a single $1$-cycle of degree $1$ and $\bmu$ is a disjoint union of $1$-cycles of degree $1$.

\begin{prop}
   \label{prop:a-union}
Suppose $\xi_1,\xi_2,\ldots,\xi_d\in\Fq^*$ and $\xi_i\neq\xi_j$ for $1\leq i\neq j\leq d$.
Let $\bla=(1)_{t-\xi_1},\bmu=(1)_{t-\xi_2}\cup\cdots\cup(1)_{t-\xi_d}$.
Then
\[
a^{\bla\cup\bmu}_{\bla\bmu}=(2q-1)^{d-1}.
\]
\end{prop}

\begin{proof}
We separate the proof in three cases.

(1). Assume $1\notin\{\xi_1,\xi_2,\ldots,\xi_d\}$.
Let
\[
\Pi=\Big\{A ~\big|~ A\sim\text{diag}\big(\xi_1^{-1},1,\ldots,1\big),
\quad
\text{diag}\big(\xi_1,\ldots,\xi_d\big)A\sim\text{diag}\Big(\xi_2,\ldots,\xi_d,1\Big) \Big\}.
\]
Since $\xi_i\neq 1$ for $1\leq i\leq d$, we have
\begin{equation}\label{eq:a-union}
a^{\bla\cup\bmu}_{\bla\bmu}
=\big|\Pi\big|.
\end{equation}
Suppose $A\in\Pi$.
Write $A=(a_{ij})_{1\leq i,j\leq d}$.
Let $\widetilde{A}=\text{diag}\big(\xi_1,\ldots,\xi_d\big)A$.
Since $A\sim\text{diag}\big(\xi_1^{-1},1,\ldots,1\big)$ and $\widetilde{A}\sim\text{diag}\big(\xi_2,\ldots,\xi_d,1\Big)$, we have
\begin{align}
\text{rank}(A-I_d)&=1,\label{eq:rankA-I}  \\
\quad a_{11}+a_{22}+\cdots+a_{dd}&=\xi_1^{-1}+(d-1),\label{eq:traceA}  \\
 \det(\widetilde{A}-\xi_iI_d)&=0, \quad \text{for }2\leq i\leq d.  \label{eq:det-wA-xiI}
\end{align}

\vspace{2mm}

\noindent {\bf Claim.}
We have $a_{22}=a_{33}=\cdots=a_{dd}=1$ and $a_{11}=\xi_1^{-1}$.

Let us prove the claim by contradiction. Assume $a_{22}\neq 1$. Then by \eqref{eq:rankA-I}, there exist $\al_1,\al_2,\al_3, \ldots,\al_d$ with $\al_2=1$ such that
\begin{equation}\label{eq:aij}
a_{ij}-\delta_{ij}=\al_i (a_{2j}-\delta_{2j}), \quad \text{ for }1\leq i,j\leq d.
\end{equation}
Since $\widetilde{A}=\text{diag}\big(\xi_1,\ldots,\xi_d\big)A$, a direct calculation using \eqref{eq:aij} shows that
$$
\det (\widetilde{A}-\xi_2I_d)=(\xi_1-\xi_2)\xi_2(a_{22}-1)(\xi_3-\xi_2)\cdots(\xi_d-\xi_2)\neq 0
$$
since $\xi_i\neq \xi_j$ for $i\neq j$.
This contradicts \eqref{eq:det-wA-xiI}. Hence $a_{22}=1$. Similarly, we can prove $a_{33}=\cdots=a_{dd}=1$.
Now it follows by \eqref{eq:traceA} that $a_{11}=\xi^{-1}$. The Claim is proved.

Observe that $\xi^{-1}_1\neq 1$. This means the first row of $A-I_d$ is nonzero and then again by \eqref{eq:rankA-I}, each $A\in\Pi$ is of the following form:
\begin{equation}\label{eq:A-generalform}
A=\begin{bmatrix}
\xi_1^{-1}          &   a_2          &      a_3          &  a_4          &  \cdots &  a_d\\
\beta_2(\xi_1^{-1}-1) &   1            &   \beta_2 a_3     &  \beta_2 a_4  &  \cdots &  \beta_2a_d\\
\beta_2(\xi_1^{-1}-1) &   \beta_3a_2   &        1          &  \beta_3 a_4  &  \cdots &  \beta_3a_d\\
\vdots              &    \vdots      &   \vdots          &   \vdots      &  \vdots &   \vdots\\
\beta_d(\xi_1^{-1}-1) &   \beta_da_2   &   \beta_d a_3     &  \beta_d a_4  &  \cdots &  1
\end{bmatrix}
\end{equation}
where $a_2,\ldots,a_d,\beta_2,\ldots,\beta_d\in\Fq$ satisify
\begin{equation}
  \label{eq:ab=0}
 a_i\beta_i=0, \qquad \text{ for }\; 2\leq i\leq d.
\end{equation}

Conversely, let $A\in M_{d\times d}(q)$ be of the form \eqref{eq:A-generalform}. Then we have
$A\sim\text{diag}\big(\xi_1^{-1},1,\ldots,1\big)$. Let $\widetilde{A}=\text{diag}\big(\xi_1,\ldots,\xi_d\big)A$.
A direct computations using \eqref{eq:ab=0} shows that
\begin{equation}\label{eq:det-wA}
\det(\widetilde{A}-\xi_iI_d)=0.
\end{equation}
Observe that the trace of $\widetilde{A}$ is $\text{tr}(\widetilde{A})=1+\xi_2+\ldots+\xi_d$.
This together with \eqref{eq:det-wA} implies $\det(\widetilde{A}-I_d)=0$.
Therefore we have $\widetilde{A}=\text{diag}\big(\xi_1,\ldots,\xi_d\big)A\sim\text{diag}\Big(\xi_2,\ldots,\xi_d,1\Big)$.
Putting these together we obtain
\[
\Pi=\big\{A~\big |~ A \text{ is of the form } \eqref{eq:A-generalform}, \text{ where } a_i,\beta_i\in\Fq \; (2\leq i\leq d)  \text{ satisfy } \eqref{eq:ab=0} \big\}.
\]
Thus, by \eqref{eq:a-union} we have
$a^{\bla\cup\bmu}_{\bla\bmu}=\big|\Pi\big|=(2q-1)^{d-1}.$

(2). Assume $\xi_1=1$.
Set
\[
\small
B=\begin{bmatrix}
 1      & -1  &  0      &  \cdots       & 0                 \\
 0      & 1  &  0      &  \cdots       &0           \\
 0      & 0  &  1      &  \cdots       & 0   \\
        &    &         &   \ddots         &    &                      \\
 0      &   0     & 0      & 0               &1
\end{bmatrix},
C=\begin{bmatrix}
 1      & 1  &  0      &  \cdots       & 0                 \\
 0      & 1  &  0      &  \cdots       &0           \\
 0      & 0  &  \xi_2     &  \cdots       & 0   \\
        &    &         &   \ddots         &    &                      \\
 0      &   0     & 0      & 0               &\xi_d
\end{bmatrix},
D=\begin{bmatrix}
 1      & 0  &  0      &  \cdots       & 0                 \\
 0      & 1  &  0      &  \cdots       &0           \\
 0      & 0  &  \xi_2      &  \cdots       & 0   \\
        &    &         &   \ddots         &    &                      \\
 0      &   0     & 0      & 0               &\xi_d
\end{bmatrix},
\]
where $B,C,D$ are $(d+1)\times (d+1)$-matrices.
Then let
\[
\Pi=\Big\{A ~\big|~ A\sim B,
\quad
CA\sim D \Big\}.
\]
Clearly
$
a_{\bla\bmu}^{\bla\cup\bmu}=|\Pi|.
$
Suppose $A\in\Pi$.
Write $A=(a_{ij})_{1\leq i,j\leq d+1}$.
Let $\widetilde{A}=CA$.
Since $A\sim B$ and $\widetilde{A}\sim D$, we have
\begin{align}
\text{rank}(A-I_{d+1})&=1,\label{eq:rankA-I-2}  \\
\quad a_{11}+a_{22}+\cdots+a_{d+1,d+1}&=d+1,\label{eq:traceA-2}  \\
 \det(\widetilde{A}-\xi_iI_{d+1})&=0, \quad \text{for }2\leq i\leq d,  \label{eq:det-wA-xiI-2}\\
 \text{rank}(\widetilde{A}-I_{d+1})&=d-1,\label{eq:rank-wA-I-2}\\
 a_{11}+a_{21}+a_{22}+\xi_2a_{33}+\cdots+\xi_da_{d+1,d+1}&=\xi_2+\cdots+\xi_d+2.\label{eq:trace-wA-2}
\end{align}
Then by a similar proof of the claim in Case (1),
we can show $a_{33}=a_{44}=\cdots=a_{d+1,d+1}=1$ using \eqref{eq:rankA-I-2} and \eqref{eq:det-wA-xiI-2}
and hence $a_{11}+a_{22}=2$ and $a_{21}=0$ by \eqref{eq:traceA-2} and \eqref{eq:trace-wA-2}. Then using \eqref{eq:rank-wA-I-2} we can deduce that
$a_{31}=a_{41}=\cdots=a_{d+1,1}=0$ and hence $a_{11}=a_{22}=1$ by \eqref{eq:rankA-I-2}. Moreover, one can show  $a_{12}+1=0$ by  \eqref{eq:det-wA-xiI-2} and \eqref{eq:rank-wA-I-2}.
This means the first row of $A-I_{d+1}$ is nonzero and again by \eqref{eq:rankA-I-2}, each $A\in\Pi$ has the form
\begin{equation}\label{eq:A-generalform-2}
A=\begin{bmatrix}
1  &  -1       &  a_3         &  a_4  &    a_5&  \cdots  &  a_{d+1}\\
0  &   1       &  0           &  0    &      0 &\cdots &   0      \\
0  & -\beta_3  &  1           &  \beta_3a_4    &    a_5 &  \cdots   &  \beta_3a_{d+1}\\
0  & -\beta_4   &  \beta_4 a_3 & 1  &  \beta_4 a_5& \cdots & \beta_4 a_{d+1}\\
\vdots & \vdots  &  \vdots   &   \vdots &   \vdots &   \vdots  &\vdots\\
0 &  -\beta_{d+1}  &  \beta_{d+1}a_3  & \beta_{d+1}a_4& \beta_{d+1}a_5& \cdots& 1
\end{bmatrix},
\end{equation}
where $a_3,a_4,\ldots,a_{d+1},\beta_3,\ldots,\beta_{d+1}\in\Fq$ satisfy
\begin{equation}\label{eq:ab=0-2}
a_i\beta_i=0,\quad \text{ for } 3\leq i\leq d+1.
\end{equation}
Again similar to the proof of Case (1), we can show that a matrix $A$ of the form \eqref{eq:A-generalform-2} satisfying \eqref{eq:ab=0-2}
belongs to $\Pi$. Hence we obtain
\[
\Pi=\{A~|~ A\text{ is of the form }\eqref{eq:A-generalform-2}, \text{ where }a_i,\beta_i\in\Fq\; (3\leq i\leq d+1) \text{ satisfy }\eqref{eq:ab=0-2}\}.
\]
Therefore we have $a_{\bla\bmu}^{\bla\cup\bmu}=|\Pi|=(2q-1)^{d-1}$.

(3). Assume $1\in\{\xi_2,\xi_3,\ldots,\xi_d\}$.
Without loss of generality, we can assume $\xi_2=1$.
Set
\[
\small
C=\begin{bmatrix}
 1      & 1  &  0      &  0  &  \cdots       & 0                 \\
 0      & 1  &  0      &   0  &   \cdots       &0           \\
 0      & 0  &  \xi_1     &  0  &   \cdots       & 0   \\
 0      & 0  &   0  &  \xi_3  &  \cdots &   0\\
        &    &         &      &\ddots         &    &                      \\
 0      &   0     & 0      & 0    &0           &\xi_d
\end{bmatrix},\quad
D=\begin{bmatrix}
 1      & 1  &  0      &  0  &  \cdots       & 0                 \\
 0      & 1  &  0      &   0  &   \cdots       &0           \\
 0      & 0  &  1     &  0  &   \cdots       & 0   \\
 0      & 0  &   0  &  \xi_3  &  \cdots &   0\\
        &    &         &      &\ddots         &       \\
 0      &   0     & 0      & 0    &0           &\xi_d
\end{bmatrix}.
\]
Then let
\[
\Pi=\Big\{A ~\big|~ A\sim \text{diag}\big(\xi_1^{-1},1,\ldots,1\big),
\quad
CA\sim D \Big\}.
\]
Clearly $a_{\bla\bmu}^{\bla\cup\bmu}=|\Pi|.$
Suppose $A\in\Pi$.
Write $A=(a_{ij})_{1\leq i,j\leq d+1}$.
Let $\widetilde{A}=CA$.
Since $A\sim \text{diag}\big(\xi_1^{-1},1,\ldots,1\big)$ and $\widetilde{A}\sim D$, we have
\begin{align}
\text{rank}(A-I_{d+1})&=1,\label{eq:rankA-I-3}  \\
\quad a_{11}+a_{22}+\cdots+a_{d+1,d+1}&=\xi_1^{-1}+d,\label{eq:traceA-3}  \\
 \det(\widetilde{A}-\xi_iI_{d+1})&=0, \quad \text{for }3\leq i\leq d,  \label{eq:det-wA-xiI-3}\\
 \text{rank}(\widetilde{A}-I_{d+1})&=d-1,\label{eq:rank-wA-I-3}\\
 a_{11}+a_{21}+a_{22}+\xi_1a_{33}+\xi_3a_{44}+\cdots+\xi_da_{d+1,d+1}&=\xi_3+\cdots+\xi_d+3.\label{eq:trace-wA-3}
\end{align}
Again by a similar proof of the claim in Case (1), we have $a_{44}=a_{55}=\cdots=a_{d+1,d+1}=1$.
Then by \eqref{eq:traceA-3} and \eqref{eq:trace-wA-3} we have
\begin{equation}\label{eq:a11-33}
a_{11}+a_{22}+a_{33}=\xi_1^{-1}+2,\quad a_{11}+a_{21}+a_{22}+\xi_1a_{33}=3,
\end{equation}
We claim $a_{21}=a_{23}=a_{24}=\cdots=a_{2,d+1}=0, a_{22}=1$. Otherwise, the second row of the matrix $A-I_{d+1}$ is nonzero and by \eqref{eq:rankA-I-3} there exist $\alpha_1,\ldots,\alpha_{d+1}$
with $\alpha_2=1$ such that
\begin{equation}\label{eq:aij-3}
a_{ij}-\delta_{ij}=\alpha_i(a_{2j}-\delta_{2j}), \quad \text{ for } 1\leq i,j\leq d+1.
\end{equation}
Then we have
\[
\text{rank}(\widetilde{A}-I_{d+1})
=\text{rank}\begin{bmatrix}
0       &    1        &   0         &  0        &   \cdots   &      0             \\
a_{21}  &  a_{22}-1   & a_{23}      &   a_{24}  &   \cdots   &      a_{2,d+1}     \\
0       &   0         &  \xi_1-1    &   0       &    \cdots  &      0             \\
0       &     0       &   0         & \xi_3-1   &    \cdots  &      0              \\
        &             &             &           &    \ddots  &                     \\
0       &   0         &   0         &   0       &    \cdots  &    \xi_d-1
\end{bmatrix}\geq d,
\]
which contradicts with \eqref{eq:rank-wA-I-3}.
So the claim holds. Then by \eqref{eq:a11-33} we obtain $a_{11}=1$ and $a_{33}=\xi_1^{-1}\neq 1$.
This means the third row of the matrix $A-I_{d+1}$ is nonzero and hence by \eqref{eq:rankA-I-3}, each $A\in\Pi$ has the following form:
\begin{equation}\label{eq:A-generalform-3}
A=\begin{bmatrix}
1      &       \beta_1a_2      &      \beta_1(\xi_1^{-1}-1)    &    \beta_1a_4    &    \cdots    & \beta_1a_{d+1}\\
0      &        1              &      0                        &    0             &     \cdots   &  0            \\
0      &        a_2            &      \xi_1^{-1}               &    a_4           &     \cdots   &  a_{d+1}      \\
0      &        \beta_4 a_2    &      \beta_4(\xi_1^{-1}-1)    &     1            &      \cdots  &  \beta_4a_{d+1}\\
       &                       &                               &                  &      \ddots  &                 \\
0      &      \beta_{d+1}a_2   &     \beta_{d+1}(\xi_1^{-1}-1) &    \beta_{d+1}a_4&     \cdots   &  1
\end{bmatrix},
\end{equation}
where $a_2,a_4,a_5,\cdots,a_{d+1},\beta_1,\beta_4,\beta_5,\ldots,\beta_{d+1}\in\Fq$ such that $\beta_ia_i=0$ for $4\leq i\leq d+1$.
Then using \eqref{eq:rank-wA-I-3} one can deduce that $\beta_1a_2=0$.
Therefore each $A\in\Pi$ has the form \eqref{eq:A-generalform-3} with
\begin{equation}\label{eq:ab=0-3}
\beta_1a_2=0,\quad \beta_ia_i=0,\quad \text{ for }4\leq i\leq d+1.
\end{equation}
Conversely, by a similar argument as for Case (1), we can show that a matrix of the form \eqref{eq:A-generalform-3} satisfying \eqref{eq:ab=0-3} must belong to $\Pi$, and hence
\[
\Pi=\big\{A~|~ A\text{ is of the form }\eqref{eq:A-generalform-3}, \text{ where }a_2,a_i,\beta_1,\beta_i\in\Fq\; (4\leq i\leq d+1) \text{ satisfy }\eqref{eq:ab=0-3} \big\}.
\]
Therefore we have $a_{\bla\bmu}^{\bla\cup\bmu}=|\Pi|=(2q-1)^{d-1}$.
The proposition is proved.
\end{proof}

For $m,b \in \N$, define the $q$-integers, $q$-factorials, and $q$-binomial coefficients
\begin{align}
 \begin{split}
  \label{eq:q-number}
[m]=[m]_q &=\frac{q^m-1}{q-1},
\qquad
[m]! =[m]_q! = [m][m-1]\cdots [1],
\\
\qbinom{m}{b} & =\qbinom{m}{b}_q =\frac{[m][m-1]\cdots [m-b+1]}{[b]!}.
 \end{split}
\end{align}
\begin{prop}\label{prop:a-union-equal}
Let $c,d\geq 1$.
Then
$a^{(1^{c+d})_{t-\xi}}_{(1^c)_{t-\xi}(1^d)_{t-\xi}}=q^{cd}\qbinom{c+d}{c}$,
if $\xi\in\Fq\setminus\{0,1\}$.
\end{prop}
\begin{proof}
Let
$$
C=\text{ diag }(\underbrace{\xi,\ldots,\xi}_{c},\underbrace{1,\ldots,1}_d),\quad
D=\text{ diag }(\underbrace{\xi,\ldots,\xi}_{d},\underbrace{1,\ldots,1}_c).
$$
Set
\[
\Pi=\{(A,B)~|~ A\sim C,\quad B=\xi A^{-1}\sim D\}.
\]
Then $a^{(1^{c+d})_{t-\xi}}_{(1^c)_{t-\xi}(1^d)_{t-\xi}}=|\Pi|$.
Observe that if a matrix $A$ is conjugate to $C$
then $\xi A^{-1}$ must be conjugate to $D$, and hence
$
\Pi=\{(A,\xi A^{-1})~|~A\sim C\}.
$
This and \eqref{eq:cc-ord} give us
\[
|\Pi|=|\mathscr K_{\bla}(c+d)|=\frac{|GL_{c+d}(q)|}{|\mathcal{A}_{\bla^{\uparrow c+d}}|}
=\frac{q^{\frac{(c+d)((c+d)-1)}{2}}[c+d]!}{q^{\frac{c(c-1)}{2}}[c]!\cdot q^{\frac{d(d-1)}{2}}[d]!}
=q^{cd}\qbinom{c+d}{c},
\]
where $\bla=(1^c)_{t-\xi}$.
The proposition is proved.
\end{proof}
\subsubsection{}

For convenience, we shall denote by $\llbracket g \rrbracket$ the class sum corresponding to $g\in G_n$.

Regarding Proposition \ref{prop:a-union}, here are some examples for $a^{\bla\cup\bmu}_{\bla\bmu}$ when $\bla$
is a single $1$-cycle of degree 1 and $\bmu$ is union of $1$-cycles of degree $1$.
\begin{example}
 \label{ex:xiequal}
(1). Suppose $q=3$. Then
\begin{align*}
\small
\left\llbracket
 \begin{array}{ccccc}
 2  &  0  &  0  &  0  &  0\\
 0  &  1  &  0  &  0  &  0\\
 0  &  0  &  1  &  0  &  0\\
 0  &  0  &  0  &  1  &  0\\
 0  &  0  &  0  &  0  &  1
\end{array}
\right \rrbracket
\cdot \left\llbracket
  \begin{array}{ccccc}
 1  &  1  &  0  &  0  &  0\\
 0  &  1  &  0  &  0  &  0\\
 0  &  0  &  1  &  1  &  0\\
 0  &  0  &  0  &  1  &  0\\
 0  &  0  &  0  &  0  &  1
\end{array}
\right \rrbracket
\small
&
=17 \left\llbracket
  \begin{array}{ccccc}
 2  &  0  &  0  &  0  &  0\\
 0  &  1  &  1  &  0  &  0\\
 0  &  0  &  1  &  0  &  0\\
 0  &  0  &  0  &  1  &  1\\
 0  &  0  &  0  &  0  &  1
\end{array}
\right \rrbracket
+\text{other terms}.\\
\small
\left\llbracket
 \begin{array}{ccccc}
 1  &  1  &  0  &  0  &  0\\
 0  &  1  &  0  &  0  &  0\\
 0  &  0  &  1  &  0  &  0\\
 0  &  0  &  0  &  1  &  0\\
 0  &  0  &  0  &  0  &  1
\end{array}
\right \rrbracket
\cdot \left\llbracket
  \begin{array}{ccccc}
 1  &  1  &  0  &  0  &  0\\
 0  &  1  &  0  &  0  &  0\\
 0  &  0  &  2  &  0  &  0\\
 0  &  0  &  0  &  1  &  0\\
 0  &  0  &  0  &  0  &  1
\end{array}
\right \rrbracket
\small
&
=60 \left\llbracket
  \begin{array}{ccccc}
 2  &  0  &  0  &  0  &  0\\
 0  &  1  &  1  &  0  &  0\\
 0  &  0  &  1  &  0  &  0\\
 0  &  0  &  0  &  1  &  1\\
 0  &  0  &  0  &  0  &  1
\end{array}
\right \rrbracket
+\text{other terms}.\\
\small
\left\llbracket
 \begin{array}{cccccc}
 2  &  0  &  0  &  0  &  0  &  0\\
 0  &  1  &  0  &  0  &  0  &  0\\
 0  &  0  &  1  &  0  &  0  &  0\\
 0  &  0  &  0  &  1  &  0  &  0\\
 0  &  0  &  0  &  0  &  1  &  0\\
 0  &  0  &  0  &  0  &  0  &  1
\end{array}
\right \rrbracket
\cdot \left\llbracket
  \begin{array}{cccccc}
 2  &  0  &  0  &  0  &  0  &  0\\
 0  &  1  &  1  &  0  &  0  &  0\\
 0  &  0  &  1  &  0  &  0  &  0\\
 0  &  0  &  0  &  1  &  1  &  0\\
 0  &  0  &  0  &  0  &  1  &  0\\
 0  &  0  &  0  &  0  &  0  &  1
\end{array}
\right \rrbracket
\small
&
=204 \left\llbracket
  \begin{array}{cccccc}
 2  &  0  &  0  &  0  &  0  &  0\\
 0  &  2  &  0  &  0  &  0  &  0\\
 0  &  0  &  1  &  1  &  0  &  0\\
 0  &  0  &  0  &  1  &  0  &  0\\
 0  &  0  &  0  &  0  &  1  &  1\\
 0  &  0  &  0  &  0  &  0  &  1
\end{array}
\right \rrbracket
+\text{other terms}.
\end{align*}

(2). Suppose $q=5$. Then
\begin{align*}
\small
\left\llbracket
 \begin{array}{ccc}
 2  &  0  &  0   \\
 0  &  1  &  0  \\
 0  &  0  &  1
\end{array}
\right \rrbracket
\cdot \left\llbracket
  \begin{array}{ccc}
 3  &  0  &  0  \\
 0  &  3  &  0 \\
 0  &  0  &  1
\end{array}
\right \rrbracket
\small
&
=49 \left\llbracket
  \begin{array}{ccc}
 2  &  0  &  0  \\
 0  &  3  &  0  \\
 0  &  0  &  3
\end{array}
\right \rrbracket
+\text{other terms}.\\
\small
\left\llbracket
 \begin{array}{cccc}
 2  &  0  &  0  &  0 \\
 0  &  1  &  0  &  0\\
 0  &  0  &  1  &  0\\
 0  &  0  &  0  &  1
\end{array}
\right \rrbracket
\cdot \left\llbracket
  \begin{array}{cccc}
 3  &  0  &  0  &  0 \\
 0  &  3  &  0  &  0\\
 0  &  0  &  3  &  0\\
 0  &  0  &  0  &  1
\end{array}
\right \rrbracket
\small
&
=249 \left\llbracket
  \begin{array}{cccc}
 2  &  0  &  0  &  0 \\
 0  &  3  &  0  &  0\\
 0  &  0  &  3  &  0\\
 0  &  0  &  0  &  3
\end{array}
\right \rrbracket
+\text{other terms}.\\
\small
\left\llbracket
 \begin{array}{cccc}
 2  &  0  &  0  &  0  \\
 0  &  1  &  0  &  0  \\
 0  &  0  &  1  &  0  \\
 0  &  0  &  0  &  1
\end{array}
\right \rrbracket
\cdot \left\llbracket
  \begin{array}{cccc}
 3  &  0  &  0  &  0  \\
 0  &  3  &  0  &  0  \\
 0  &  0  &  4  &  0  \\
 0  &  0  &  0  &  1
\end{array}
\right \rrbracket
\small
&
=441 \left\llbracket
  \begin{array}{cccc}
 2  &  0  &  0  &  0  \\
 0  &  3  &  0  &  0  \\
 0  &  0  &  3  &  0  \\
 0  &  0  &  0  &  4
\end{array}
\right \rrbracket
+\text{other terms}.\\
\small
\left\llbracket
 \begin{array}{cccc}
 4  &  0  &  0  &  0  \\
 0  &  1  &  0  &  0 \\
 0  &  0  &  1  &  0  \\
 0  &  0  &  0  &  1
\end{array}
\right \rrbracket
\cdot \left\llbracket
  \begin{array}{cccc}
 4  &  0  &  0  &  0  \\
 0  &  3  &  0  &  0 \\
 0  &  0  &  3  &  0  \\
 0  &  0  &  0  &  1
\end{array}
\right \rrbracket
\small
&
=1470 \left\llbracket
  \begin{array}{cccc}
 4  &  0  &  0  &  0  \\
 0  &  4  &  0  &  0 \\
 0  &  0  &  3  &  0  \\
 0  &  0  &  0  &  3
\end{array}
\right \rrbracket
+\text{other terms}.
\end{align*}
\end{example}

%

\subsection{More examples}

Here we present examples for $a_{\bla\bmu}^\bnu$, where $\bla=(1)_{t-\xi}$ for $\xi \in \Fq\backslash \{0,1\}$, $\bmu =(1)_{f'}$ for $f'\in\Phi$ of degree $2$, and $\bnu =(1)_f$ for $f\in\Phi$ of degree $3$.

\begin{example}\label{example:q3}
Suppose $q=3$.
Then
\begin{align*}
&\left\llbracket
 \begin{array}{ccc}
 2  &  0  &  0\\
 0  &  1  &  0\\
 0  &  0  &  1
\end{array}
\right \rrbracket
\cdot \left\llbracket
 \begin{array}{ccc}
 0  &  1  &  0\\
 1  &  2  &  0\\
 0 &  0  &  1
\end{array}
\right \rrbracket \\
&=13 \left\llbracket
 \begin{array}{ccc}
 0  &  1  &  0\\
0  &  0  &  1\\
 1  &  1 &  0
\end{array}
\right \rrbracket
+13 \left\llbracket
 \begin{array}{ccc}
 0  &  1  &  0\\
0  &  0  &  1\\
 1  &  0 &  2
\end{array}
\right \rrbracket
+13 \left\llbracket
 \begin{array}{ccc}
0  &  1  &  0\\
0  &  0  &  1\\
 1  &  2 & 2
\end{array}
\right \rrbracket
+13 \left\llbracket
 \begin{array}{ccc}
 0  &  1  &  0\\
0  &  0  &  1\\
 1  &  1 &  1
\end{array}
\right \rrbracket \\
&+\text{other terms}.
\end{align*}
\end{example}

\begin{example}\label{example:q5}
Suppose $q=5$.
Then
\begin{align*}
&\tiny
\left\llbracket
 \begin{array}{ccc}
 2  &  0  &  0\\
 0  &  1  &  0\\
 0  &  0  &  1
\end{array}
\right \rrbracket
\cdot \left\llbracket
 \begin{array}{ccc}
 0  &  1  &  0\\
 3  &  1  &  0\\
 0  &  0  &  1
\end{array}
\right \rrbracket \\
&
\tiny
=31 \left\llbracket
 \begin{array}{ccc}
 0  &  1  &  0\\
0  &  0  &  1\\
 4  &  4  &  0
\end{array}
\right \rrbracket
+31 \left\llbracket
 \begin{array}{ccc}
 0  &  1  &  0\\
0  &  0  &  1\\
 4  &  3  &  0
\end{array}
\right \rrbracket
+31 \left\llbracket
 \begin{array}{ccc}
 0  &  1  &  0\\
0  &  0  &  1\\
 4  &  0  &  4
\end{array}
\right \rrbracket
+31 \left\llbracket
 \begin{array}{ccc}
 0  &  1  &  0\\
0  &  0  &  1\\
 4  &  2  &  4
\end{array}
\right \rrbracket
+31 \left\llbracket
 \begin{array}{ccc}
 0  &  1  &  0\\
0  &  0  &  1\\
 4  &  1  &  4
\end{array}
\right \rrbracket \\
&\tiny
+31 \left\llbracket
 \begin{array}{ccc}
 0  &  1  &  0\\
0  &  0  &  1\\
 4  &  0  &  3
\end{array}
\right \rrbracket
+31 \left\llbracket
 \begin{array}{ccc}
 0  &  1  &  0\\
0  &  0  &  1\\
 4  &  4  &  2
\end{array}
\right \rrbracket
+31 \left\llbracket
 \begin{array}{ccc}
 0  &  1  &  0\\
0  &  0  &  1\\
 4  &  1  &  2
\end{array}
\right \rrbracket
+31 \left\llbracket
 \begin{array}{ccc}
 0  &  1  &  0\\
0  &  0  &  1\\
 4  &  4  &  1
\end{array}
\right \rrbracket
+31 \left\llbracket
 \begin{array}{ccc}
 0  &  1  &  0\\
0  &  0  &  1\\
 4  &  2  &  1
\end{array}
\right \rrbracket \\
&+\text{other terms}.
\end{align*}
\end{example}

Note that $\det
 \begin{bmatrix}
 2  &  0  &  0\\
 0  &  1  &  0\\
 0  &  0  &  1
\end{bmatrix}
=2$,  and $
\det
 \begin{bmatrix}
 0  &  1  &  0\\
 3  &  1  &  0\\
 0  &  0  &  1
\end{bmatrix}
=-3$.
Observe that all irreducible polynomial in $\mathbb F_5[t]$ of degree $3$ with constant term equal to $-2\times (-3)=1$
appear on the right hand side of the above equation.

\section{Conjectures and discussions}
  \label{sec:Conj}

Motivated by the examples computed in the previous section, we formulate in this section several conjectures on the structure constants of the stable center, and discuss various problems arising from this work.

\subsection{Conjectures}

We present several conjectures on the structure constants $a_{\bla\bmu}^\bnu$ and the structure of the stable center $\mathbb Q\otimes_{\Z} \mathscr{G}$. Recall the $q$-integers $[m]$ from \eqref{eq:q-number}.


\begin{conjecture}
   \label{conj:str-const}
(1). Suppose $\xi_1,\xi_2,\ldots,\xi_d\in\Fq^*$ are distinct, and let $c_1,\ldots,c_d\in\N$ for $1\leq i\leq d$. Let $\bla=(1)_{t-\xi_1},\bmu=(1^{c_1})_{t-\xi_1}(1^{c_2})_{t-\xi_2}\cup\cdots\cup(1^{c_d})_{t-\xi_d}$.
Then
\[
a^{\bla\cup\bmu}_{\bla\bmu}=q^{c_1}[c_1+1]\prod^d_{i=2}(2q^{c_i}-1).
\]

(2).
Suppose $\xi\in\Fq^*$ and $f\in\Phi$.
Let $\bla=(1)_{t-\xi}$ and $\bmu=(1)_f$. Recall $q_f$ from \eqref{eq:qf}.
Then
$
a^{\bla\cup\bmu}_{\bla\bmu}=2q_f -1.
$

(3). Suppose $\xi\in\Fq^*, f_2,\ldots,f_d\in\Phi$ and $c_1,c_2,\ldots,c_d\in\N$ with $f_i\neq t-\xi$ for $2\leq i\leq d$.
Let $\bla=(1)_{t-\xi},\bmu=(1^{c_1})_{t-\xi}(1^{c_2})_{f_2}\cup\cdots\cup(1^{c_d})_{f_d}$.
Then
\[
a^{\bla\cup\bmu}_{\bla\bmu}=q^{c_1}[c_1+1]\prod^d_{i=2}(2(q_{f_i})^{c_i}-1).
\]
\end{conjecture}
Conjecture~\ref{conj:str-const}(1) is supported by Example~\ref{ex:xiequal}. Note that
in Example~\ref{ex:xiequal} we have $17=2q^2-1,\;60=q(1+q)(2q-1)$ and $204=q(1+q)(2q^2-1)$ with $q=3$
and $49=2q^2-1,\; 249=2q^3-1, \;441=(2q-1)(2q^2-1)$ and $1470=q(1+q)(2q^2-1)$ with $q=5$.
Conjecture~\ref{conj:str-const}(2) is supported by Lemma~\ref{lem:xi=eta=1} and Proposition~\ref{prop:a-union}.
Conjecture~\ref{conj:str-const}(3) is a combination of Conjecture~\ref{conj:str-const}(1)-(2) and
it is supported by some further examples which we omit here.

\begin{conjecture}\label{conj:(1)f}
Suppose $\bla=(1)_{t-\xi}$ for $\xi \in \Fq^*$, $\bmu =(1)_{f'}$ for $f'\in\Phi$ of degree $d-1$ with constant term  $a$.

(1). For each irreducible polynomial $f\in\Fq[t]$ of degree $d$ with constant term equal to $-\xi a$,
there exist elements $g,h\in G_d$ of modified types $\bla$ and $\bmu$, respectively, such that the modified type of $gh$
is $(1)_{f}$.

(2). Suppose $f\in\Phi$ of degree $d\ge 3$ with constant term $b$. If $b=-\xi a$, then
$
a^{(1)_f}_{\bla\bmu}=[d].
$
\end{conjecture}
Conjecture \ref{conj:(1)f}(1) holds when $\|\bmu\|=1$ by Theorem~\ref{thm:2reflections}.
Conjecture~\ref{conj:(1)f}(2) is supported by Theorem~\ref{thm:2reflections}, where $\deg f=2$, and it is also supported by Examples~ \ref{example:q3} and \ref{example:q5}.
Note in Example~ \ref{example:q3} we have $13=q^2+q+1$ with $q=3$, while in Example~ \ref{example:q5} we have $31=q^2+q+1$ with $q=5$. 

%

It will be interesting to understand how the structure constants $a^{\bnu}_{\bla\bmu}$ depend on $q$ as $q$ varies. To that end, write $\Phi_q=\Phi$ below to indicate its dependence on $q$. For each $\bla\in\mathcal{P}(\Phi_q)$, we set
$$
\Phi_q(\bla)=\{f\in\Phi_q~|~ \bla(f)\neq\emptyset\},
$$
and call it the support of $\bla$. Denote by $\Phi_\mathbb{Z}$ the set of monic irreducible polynomials in $\mathbb{Z}[t]$ other than $t$. The support $\Phi_{\mathbb{Z}}(\bla)$ is defined similarly for each $\bla\in\mathcal{P}(\Phi_{\mathbb{Z}})$. We shall regard a polynomial in $\Z[t]$ as a polynomial in $\Fq[t]$ by reduction modulo $q$. Observe that for each $f(t)\in\Phi_{\mathbb{Z}}$ we have $f(t)\in\Phi_q$ for $q$ any power of a large enough prime. Thus any $\bla\in \mathcal{P}(\Phi_{\mathbb{Z}})$ with  its support $\Phi_{\mathbb{Z}}(\bla)\subset\Phi_q$ can be viewed an element in $\mathcal{P}(\Phi_q)$.  Our next conjecture concerns about a generic version of the structure constants $a^\bnu_{\bla\bmu}$ of $\GG(q)$ as $q$ varies.

\begin{conjecture}\label{conj:intpoly}
Suppose $\bla,\bmu,\bnu\in\mathcal{P}(\Phi_\mathbb{Z})$.
There exists an integer polynomial in one variable $\LL$, $A_{\bla\bmu}^{\bnu}(\LL)\in\Z[\LL]$, such that
$a^\bnu_{\bla\bmu}=A_{\bla\bmu}^{\bnu}(q)$, for each prime power $q$ with
$\Phi_{\mathbb{Z}}(\bla),\Phi_{\mathbb{Z}}(\bmu),\Phi_{\mathbb{Z}}(\bnu)\subset\Phi_q$.
Moreover define $\texttt A^\bnu_{\bla\bmu}\in\Z[\LL]$ such that $\texttt A^\bnu_{\bla\bmu}(\LL)=A^\bnu_{\bla\bmu}(\LL+1)$, i.e., $a^\bnu_{\bla\bmu}=\texttt A_{\bla\bmu}^{\bnu}(q-1)$ for all prime powers $q$ as above. Then the following positivity holds: $\texttt A^\bnu_{\bla\bmu}\in\N[\LL]$.
\end{conjecture}
Conjecture~\ref{conj:intpoly} is supported by all examples computed in Section~\ref{sec:computation}.

We formulate the following conjecture on the stable center. Let $\mathbb Q$ denote the field of rational numbers.
\begin{conjecture}
 \label{conj:generators}
The stable center $\mathbb Q\otimes_{\Z} \mathscr{G}(q)$ is the polynomial algebra generated by the single cycle class sums $K_{(r)_f}$, for all $r\ge 1$ and $f\in \Phi$.
\end{conjecture}
In the setting of symmetric groups \cite{FH59} and wreath products \cite{W04}, the stable center (after a base change from $\Z$ to $\mathbb Q$) is known to be the polynomial algebras generated by the single cycles. As $\mathbb Q\otimes_{\Z} \mathscr{G}(q)$ has the size of the ring of symmetric functions indexed by $\Phi$, we can ask for a symmetric function interpretation of the class sum basis $\{K_\bla\}$. (This was achieved for the stable center of the symmetric groups in \cite{Ma95}.)

\subsection{Further directions}
  \label{subsec:discussions}

There are several further directions and problems arising from our work which may be worth pursuing.

\begin{enumerate}
\item
It is challenging but important to compute more examples of the structure constants $a_{\bla\bmu}^\bnu$, in particular when $\bla$ is a general single cycles, i.e., $\bla =(r)_f$ for all $r,f$, and $\bnu=\bla\cup\bmu$. It is likely that Conjecture~\ref{conj:generators} would follow from such detailed information.
\item
We ask for similar stability phenomena for other infinite families of finite groups of Lie type, such as unitary, symplectic, or orthogonal groups.
\item
We ask for similar stability phenomena for various families of subgroups of $\GL$, including the affine groups.
\item
The associated graded of the center of the complex group algebra of wreath product $\Gamma\wr S_n$ for a subgroup $\Gamma$ of $SL_2(\C)$ is isomorphic to the cohomology ring of Hilbert scheme of $n$ points on the minimal resolution $\widetilde{\C^2/\Gamma}$; see \cite{W04} (also cf. \cite{LS01, Va01, LQW04}). It can also be regarded as the Chen-Ruan orbifold cohomology ring of the orbifold $\C^{2n}\big/\Gamma\wr S_n$. Does the graded algebra $\mathscr{G}_n(q)$ in this paper afford similar geometric interpretation and generalization?

\item
Let $\bla,\bmu,\bnu,\boldsymbol{\tilde\la}, \boldsymbol{\tilde\mu}, \boldsymbol{\tilde\nu} \in \mP(\Phi)$ with $\|\bnu\|=\|\bla\|+\|\bmu\|$.
Assume that there exists a degree preserving bijection $\Phi(\bla)\cup\Phi(\bmu)\cup \Phi(\bnu)\leftrightarrow \Phi(\boldsymbol{\tilde\la})\cup\Phi(\boldsymbol{\tilde\mu})\cup \Phi(\boldsymbol{\tilde\nu})$, $f\mapsto \tilde f$ ($\deg \tilde f =\deg f$), such that
$\bla(f)= \boldsymbol{\tilde\la}(\tilde f), \bmu(f)= \boldsymbol{\tilde\mu}(\tilde f), \bnu(f)= \boldsymbol{\tilde\nu}(\tilde f)$, for all $f$, $\tilde f$.
Then from all the examples we have computed, the structure constants $a^{\bnu}_{\bla\bmu}$ only depend on the configurations of $\bla, \bmu, \bnu$ in the sense that $a^{\bnu}_{\bla\bmu} =a^{\boldsymbol{\tilde\nu}}_{\boldsymbol{\tilde\la}\boldsymbol{\tilde\mu}}$. We ask if this remarkable phenomenon holds in general.

\item
Regarding the structure constants $a_{\bla\bmu}^\bnu(n)$ in \eqref{eq:a(n)} for the center $\ZZ_n(q)$ with $ \bla,\bmu,\bnu\in\mathcal{P}(\Phi_q)$,
\cite[Theorem 3.7]{M14} states that there exist polynomials $p^{\bnu}_{\bla\bmu}(x)$ with rational coefficients  such that $a^{\bnu}_{\bla\bmu}(n)=p^\bnu_{\bla\bmu}(q^n)$; see however Remark \ref{rem:Meliot}. 
(The parametrization {\em loc. cit.} used $\obmuk$ as in \eqref{eq:obmu}, and we can replace them by the modified type $\bmu$ etc. here.) Thanks to $q^n=(q-1)[n]_q+1$,  
this can be reformulated as that there exist polynomials $\mathfrak p^{\bnu}_{\bla\bmu}(x)$ with rational coefficients  such that $a^{\bnu}_{\bla\bmu}(n)=\mathfrak p^\bnu_{\bla\bmu}([n]_q)$. We ask if the following integrality holds: $\mathfrak p^{\bnu}_{\bla\bmu}(x)\in \mathbb Z[x]$ for any $\bla, \bmu, \bnu$.

In light of Conjecture~\ref{conj:intpoly}, we ask whether there exist polynomials $\phi^{\bnu}_{\bla\bmu}(\LL, x)\in\mathbb Q(\LL)[x]$, for $\bla,\bmu,\bnu\in\mathcal{P}(\Phi_{\mathbb Z})$, such that
$a^{\bnu}_{\bla\bmu}(n)=\phi^{\bnu}_{\bla\bmu}(q,[n]_q)$ for 
any prime power $q$ with $\Phi_{\mathbb{Z}}(\bla),\Phi_{\mathbb{Z}}(\bmu),\Phi_{\mathbb{Z}}(\bnu)\subset\Phi_q$.
Furthermore, we ask if the polynomials $\phi^{\bnu}_{\bla\bmu}(\LL, x)$ are $\mathbb Z[\LL]$-linear combination of $\ds\frac{x(x-[1]_\LL)\cdots(x-[k-1]_\LL)}{\LL^{\binom{k}{2}} [k]_\LL!} $ for $k\ge 0$; cf. \cite[Proposition~ 1.2]{HH17}. For each fixed $n$, we ask if there exist positive integer polynomials $\texttt A^\bnu_{\bla\bmu;n} (\LL)\in\N[\LL]$ such that $\texttt A^\bnu_{\bla\bmu;n}(q-1)=a^\bnu_{\bla\bmu}(n)$ for all prime powers $q$ as above. This generalizes Conjecture~\ref{conj:intpoly}.


%
\end{enumerate}


\end{document}